\documentclass[11pt]{amsart}
\title{Genus-One Mirror Symmetry in the Landau-Ginzburg Model}
\author{Shuai Guo}
\address{\newline School of Mathematical Sciences\newline Peking University
 \newline No 5. Yiheyuan Road\newline  Beijing 100871 China}
\email{guoshuai@math.pku.edu.cn}
\author{Dustin Ross}
\address{\newline Department of Mathematics\newline San Francisco State University\newline Thornton Hall 941\newline 1600 Holloway Avenue\newline San Francisco, California 94132 USA}
\email{rossd@sfsu.edu}

\usepackage {color,graphicx,psfrag, verbatim,amssymb,amscd,enumerate,subfigure}
\usepackage{texdraw,xypic,appendix}
\usepackage[margin=1in]{geometry}
\input xy
\xyoption{all}

\newcommand{\GIT}{\mkern-5mu\mathbin{/\mkern-6mu/\mkern-2mu}}
\newcommand{\rk}{\textrm{rk}}
\newcommand{\ord}{\text{ord}}

\newcommand{\bP}{\mathbb{P}}

\newcommand{\N}{\mathbb{N}}
\newcommand{\Z}{\mathbb{Z}}

\newcommand{\D}{\mathcal{D}}
\newcommand{\M}{\overline{\mathcal{M}}}
\newcommand{\CM}{\overline{\mathcal{CM}}}
\newcommand{\cO}{\mathcal{O}}
\newcommand{\bQ}{\mathbb{Q}}
\newcommand{\bC}{\mathbb{C}}

\renewcommand{\L}{\mathcal{L}}

\newcommand{\vir}{\mathrm{vir}}
\newcommand{\val}{\mathrm{val}}
\newcommand{\Aut}{\mathrm{Aut}}
\newcommand{\Contr}{\mathrm{Contr}}
\newcommand{\ch}{\mathrm{ch}}
\newcommand{\mov}{\mathrm{mov}}
\newcommand{\ev}{\mathrm{ev}}

\newcommand{\CR}{\mathrm{CR}}
\newcommand{\T}{\mathbb{T}}

\renewcommand{\ev}{\mathrm{ev}}
\newcommand{\q}{\mathfrak{q}}
\renewcommand{\t}{\mathfrak{t}}
\newcommand{\Res}{\mathrm{Res}}
\newcommand{\re}{\mathrm{e}}
\newcommand{\ri}{\mathrm{i}}

\newcommand{\w}{\mathrm{w}}
\renewcommand{\S}{\mathbb{S}}
\renewcommand{\d}{\diamondsuit}
\newcommand{\h}{\heartsuit}

\newcommand{\s}{\square}
\newcommand{\E}{\mathcal{E}}
\renewcommand{\P}{\mathcal{P}}
\newcommand{\I}{I}
\newcommand{\U}{\mathbb{U}}

\newtheorem{dummy}{}[section]
\newtheorem{lemma}[dummy]{Lemma}
\newtheorem{proposition}[dummy]{Proposition}
\newtheorem{theorem}[dummy]{Theorem}

\newtheorem*{theorem1}{Main Result}

\theoremstyle{definition}

\newtheorem{definition}[dummy]{Definition}
\newtheorem{example}[dummy]{Example}

\newtheorem{remark}[dummy]{Remark}

\usepackage{graphicx}

2
\begin{document}

\maketitle
\begin{abstract}
We prove an explicit formula for the genus-one Fan-Jarvis-Ruan-Witten invariants associated to the quintic threefold, verifying the genus-one mirror conjecture of Huang, Klemm, and Quackenbush. The proof involves two steps. The first step uses localization on auxiliary moduli spaces to compare the usual Fan-Jarvis-Ruan-Witten invariants with a semisimple theory of twisted invariants. The second step uses the genus-one formula for semisimple cohomological field theories to compute the twisted invariants explicitly.
\end{abstract}

\section{Introduction}

This paper studies the genus-one Fan-Jarvis-Ruan-Witten invariants associated to the quintic threefold, which encode the degree of the Witten class on moduli spaces of $5$-spin curves. Let $F_1(\tau)$ be the restriction of the genus-one Fan-Jarvis-Ruan-Witten potential to the small state-space, where
\[
\tau(t)=\frac{I_1(t)}{I_0(t)}
\]
is the mirror map determined by genus-zero mirror symmetry. Our main theorem is the following.

\begin{theorem1}
We have
\[
F_1(\tau) = \log \left(I_0(t)^{-\frac{31}{3}}(1-({t}/{5})^{5}  )^{-\frac{1}{12}} \tau'(t)^{-\frac{1}{2}} \right).
\]
\end{theorem1}

\noindent This theorem settles the genus-one mirror conjecture of Huang--Klemm--Quackenbush \cite{HKQ}.

\subsection{Context and motivation}

In the seminal paper \cite{Witten}, Witten proposed studying phase transitions in the gauged linear sigma model. In general, phase transitions relate different phases of sigma models associated to certain geometries, a special case of which is the Landau-Ginzburg/Calabi-Yau correspondence. Mathematically, the Landau-Ginzburg/Calabi-Yau correspondence can be interpreted as an equivalence between the Gromov-Witten invariants  of a degree-$d$ hypersurface in projective space (the Calabi-Yau side) and the associated Fan-Jarvis-Ruan-Witten invariants, defined by certain intersection numbers the in moduli spaces of $d$-spin curves (the Landau-Ginzburg side).

Of particular interest is the case of the Fermat quintic threefold. The genus-zero Gromov-Witten invariants of the quintic were first computed by Givental \cite{GiventalMirror} and Lian--Liu--Yau \cite{LLY}, wherein they verified the celebrated genus-zero mirror theorem of Candelas--de la Ossa--Green--Parkes \cite{CdlOGP}. A decade later, the genus-zero Fan-Jarvis-Ruan-Witten invariants were computed by Chiodo--Ruan \cite{CR}, in which they verified an analogous genus-zero mirror theorem at the Landau-Ginzburg limit of the B-model moduli space. By analytically continuing along a path in the B-model moduli space, Chiodo and Ruan compared the two genus-zero mirror formulas and provided a mathematically precise statement of the genus-zero Landau-Ginzburg/Calabi-Yau correspondence. Following ideas of Givental \cite{GiventalSymplectic}, they formulated their result in terms of an explicit symplectic transformation $\U$ between two infinite-dimensional symplectic vector spaces associated to the respective theories.

An important aspect of the symplectic formulation of Chiodo and Ruan is that it provided a hint at how to formulate a higher-genus correspondence purely in terms of genus-zero data. More specifically, Chiodo and Ruan conjectured that the all-genus Gromov-Witten partition function is obtained from the all-genus Fan-Jarvis-Ruan-Witten partition function by the action of the geometric quantization of $\U$. If true, the higher-genus correspondence gives an explicit formula for higher-genus Gromov-Witten invariants in terms of Fan-Jarvis-Ruan-Witten invariants. Until now, however, there has not been any evidence for the higher-genus correspondence.

Shortly before Chiodo and Ruan formulated the Landau-Ginzburg/Calabi-Yau correspondence, the genus-one Gromov-Witten invariants of the quintic were computed by Zinger \cite{Zinger}, verifying the genus-one mirror conjecture of Bershadsky--Cecotti--Ooguri--Vafa \cite{BCOV}. It was also conjectured by Huang--Klemm--Quackenbush \cite{HKQ} that an analogous genus-one formula holds for the Fan-Jarvis-Ruan-Witten invariants.

The main result of this paper verifies the genus-one mirror symmetry formula conjectured by Huang, Klemm, and Quackenbush. In the sequel to this paper \cite{GR}, we use the genus-one mirror formulas contained here and in Zinger's work \cite{Zinger} to prove the genus-one specialization of Chiodo and Ruan's higher-genus Landau-Ginzburg/Calabi-Yau correspondence. This provides the first evidence for the validity of the higher-genus quantization conjecture.

\subsection{Precise statements of results}

Let $\M_{g,\vec m}^{1/5}$ denote the moduli space of stable $5$-spin curves with $n$ orbifold marked points having multiplicities $\vec m = (m_1,\dots,m_n)$. More precisely, a point in $\M_{g,\vec m}^{1/5}$ parametrizes a tuple $(C,q_1,\dots,q_n,L,\kappa)$ where
\begin{itemize}
\item $(C,q_1,\dots,q_n)$ is a stable orbifold curve with $\mu_5$ orbifold structure at all marks and nodes;
\item $L$ is an orbifold line bundle on $C$ and the $\mu_5$-representation $L|_{q_i}$ is multiplication by $\re^{2\pi\ri m_i/5}$;
\item $\kappa$ is an isomorphism
\[
\kappa:L^{\otimes 5}\cong\omega_{C,\log}.
\]
\end{itemize}
In the introduction, we take $m_i\in\{1,\dots,4\}$, though we also consider the case $m_i=5$ in the main body of the paper.

Associated to the Fermat quintic polynomial in five variables, there is a relative two-term obstruction theory on $\M_{g,\vec m}^{1/5}$ given by $R\pi_*\L^{\oplus 5}$, where $\L$ is the universal line bundle and $\pi$ the projection from the universal curve. This relative obstruction theory can be used to equip the moduli space with two different `virtual fundamental classes'. On the one hand, we have the Witten class (c.f. \cite{PV,ChiodoW,FJR,CLL})\footnote{The Witten class, here, is the cosection localized construction of Chang--Li--Li \cite{CLL}, which differs from that of Fan--Jarvis--Ruan \cite{FJR} by a sign $(-1)^{h^0(L^{\oplus 5})-h^1(L^{\oplus 5})}$.}, which we denote by
\[
\left[\M_{g,\vec m}^{1/5}\right]^{\w}.
\]
On the other hand, we can make the relative obstruction theory equivariant with respect to $\S=(\bC^*)^5$ by letting the five factors of $\S$ scale the five different copies of $\L$. By capping the usual fundamental class against the inverse equivariant Euler class of the two-term obstruction theory (this makes sense because the equivariant Euler class is multiplicative and invertible), we obtain an equivariant class, which we denote by
\[
\left[\M_{g,\vec m}^{1/5}\right]^\lambda,
\]
where $\lambda=(\lambda_1,\dots,\lambda_5)$ denotes the equivariant parameters.

Let $\phi_k$ where $k=0,\dots,3$ be formal symbols. We have two types of correlators corresponding to the two virtual classes: for $\star=\w$ or $\lambda$, define
\begin{equation}\label{eq:correlators}
\left\langle\phi_{m_1-1}\psi^{a_1}\cdots\phi_{m_n-1}\psi^{a_n} \right\rangle^{\star}_{g,n}:=\frac{(-1)^{3-3g+n-\sum_{i}m_i}}{5^{2g-2}}\int_{\left[\M_{g,\vec m}^{1/5}\right]^\star}\psi_1^{a_1}\cdots\psi_n^{a_n}
\end{equation}
where the $\psi_i$ are the cotangent-line classes on the coarse curve. The $\w$-correlators are typically called (narrow) \emph{Fan--Jarvis--Ruan--Witten (FJRW) invariants}, due to their development in full generality by Fan--Jarvis--Ruan \cite{FJR}, while the $\lambda$-correlators are a particular type of \emph{twisted $5$-spin invariants}. The sign convention is simply to maintain consistency with the original definitions of Fan--Jarvis--Ruan \cite{FJR}.\footnote{In \cite{FJR}, the factor $5^{2g-2}$ in \eqref{eq:correlators} is $5^{g-1}$. We choose to alter this factor in order to make the correlators more consistent with the Gromov-Witten invariants of the quintic 3-fold.} 

The correlators \eqref{eq:correlators} can be used to define cohomological field theories (CohFTs). Our primary interest in this work is to study the (genus-one part of the) FJRW CohFT associated to the Witten class. However, the twisted CohFT has the distinct advantage of generic semi-simplicity, which, by results of  Dubrovin--Zhang \cite{DZ}, Givental \cite{Giv}, and Teleman \cite{Teleman}, implies that the higher-genus invariants can be reconstructed from the genus-zero invariants. With this motivation in mind, the current paper contains two distinct parts: the first part provides a genus-one comparison between the FJRW and twisted correlators, while the second part uses semisimple reconstruction to provide an explicit computation of the genus-one twisted correlators. Together, they imply the genus-one mirror symmetry theorem for FJRW invariants.

\subsubsection{Part one -- a comparison result} In genus-zero, the relationship between FJRW and twisted correlators is simple. Namely, the FJRW correlators can be obtained from the twisted correlators by specializing $\lambda=0$. The explanation for this is not difficult: in genus zero, one of the terms in the two-term obstruction theory vanishes. Thus, the complex represents a vector bundle and the Euler class of that vector bundle \emph{is} the Witten class.

In higher genera, the obstruction theory is an honest two-term complex and we can no longer take such a non-equivariant limit. However, in genus one, the situation is not so bad.  The genus-one FJRW invariants are completely determined by the correlators
\[
\left\langle\phi_1\cdots\phi_1 \right\rangle_{1,n}^\w,
\]
where all orbifold points have multiplicity two. A key fact about the underlying moduli spaces $\M_{1,(2,\dots,2)}^{1/5}$ is that the locus where the obstruction theory fails to be a vector bundle is a sublocus of rational tails. Therefore, if we can find a way to eliminate rational tails, the obstruction complex will again represent a vector bundle whose Euler class is the Witten class.

Following ideas of Ciocan-Fontanine--Kim \cite{CFK,CFK2,CFK4} and Ross--Ruan \cite{RR}, we know that a reasonable way to eliminate rational tails in this setting is through `wall-crossing' techniques. More specifically, let $\M_{g,\vec m|\delta}^{1/5,\epsilon}$ be the moduli space of $5$-spin curves with $\delta$ additional indistinguishable weight-$\epsilon$ points of type $\phi_1$. In other words, a point in $\M_{g,\vec m|\delta}^{1/5,\epsilon}$ parametrizes a tuple $(C,q_1,\dots,q_n,L,D,\kappa)$ where
\begin{itemize}
\item $(C,q_1,\dots,q_n)$ is an orbifold curve with $\mu_5$ orbifold structure at all marks and nodes;
\item $D$ is an effective divisor on $C$, disjoint from the nodes and marks, with $|D|=\delta$ and such that $\deg_x(D)\leq1/\epsilon$ for all $x\in C$.
\item The tuple $(C,q_1,\dots,q_n,D)$ is $\epsilon$-stable; i.e.
\[
\omega_{C,\log}\otimes\cO(\epsilon D)
\]
is ample.
\item $L$ is an orbifold line bundle on $C$ and the $\mu_5$-representation $L|_{q_i}$ is multiplication by $\re^{2\pi\ri m_i/5}$;
\item $\kappa$ is an isomorphism
\[
\kappa:L^{\otimes 5}\cong\omega_{C,\log}\otimes\cO(-D).
\]
\end{itemize}
When $\epsilon=\infty$, we simply recover $\M_{g,\vec m}^{1/5}$. We denote by $\M_{g,\vec m}^{1/5,0}$ the limit as $\epsilon$ tends to zero.

In regards to the earlier discussion, a key observation at this point is that rational tails are completely disallowed in $\M_{g,\emptyset|\delta}^{1/5,0}$, as can easily be seen by the $\epsilon$-stability condition. As before, there are two types of virtual fundamental classes for $\star=\w$ or $\lambda$, and we define correlators:
\begin{equation*}
\left\langle\phi_{m_1-1}\psi^{a_1}\cdots\phi_{m_n-1}\psi^{a_n} \right\rangle^{\star,\epsilon}_{g,n|\delta}:=\frac{(-1)^{3-3g+n-\delta-\sum_{i}m_i}}{5^{2g-2}}\int_{\left[\M_{g,\vec m|\delta}^{1/5,\epsilon}\right]^\star}\psi^a_1\cdots\psi_n^{a_n}.
\end{equation*}

In order to state our genus-one wall-crossing formulas explicitly, we recall the FJRW $I$-function:\footnote{We warn the reader that this I-function differs from the one in \cite{CR} by a factor of $t$.}
\begin{equation}\label{eq:FJRWIfunction}
I(t,z):=z\sum_{a\geq 0}\frac{t^{a}}{z^aa!}\prod_{0\leq k<\frac{a+1}{5} \atop \langle k \rangle = \langle \frac{a+1}{5}\rangle}(kz)^5\phi_{a}.
\end{equation}
Let $\tau  = \frac{I_1(t)}{I_0(t)}$ denote the mirror map, where the series $I_0(t)$ and $I_1(t)$ are defined by considering the expansion of $I(t,z)$ as a Laurent series in $z^{-1}$:
\[
I(t,z)=:I_0(t)z\phi_0+I_1(t)\phi_1+\cO(z^{-1}).
\]

The following result provides a precise way in which we can `remove rational tails' in both the FJRW and the twisted setting.

\begin{theorem}[Genus-one wall-crossing]\label{thm:wc}
For $\star=\w$ or $\lambda$,
\[
\sum_{n>1}\frac{1}{n!}\left\langle\left(\tau\phi_{1} \right)^n \right\rangle^{\star,\infty}_{1,n}=\log(I_0(t)) \langle \phi_0\psi_1 \rangle^{\star,\infty}_{1,1}+\sum_{\delta>1}t^\delta\left\langle - \right\rangle^{\star,0}_{1,0\mid \delta}.
\]
\end{theorem}

The proof of this theorem is obtained by manipulating certain localization relations that have appeared in recent work of Chang--Li--Li--Liu \cite{CLLL2}. Since the correlators in the final term of Theorem \ref{thm:wc} are defined over moduli spaces of genus-one $5$-spin curves without marks or rational tails, the obstruction complex represents a bundle and it follows that
\[
\left\langle - \right\rangle^{\w,0}_{1,0\mid \delta}=\left\langle - \right\rangle^{\lambda,0}_{1,0\mid \delta}.
\]
From this observation, we obtain the following comparison between the FJRW correlators and the twisted correlators.

\begin{theorem}[Genus-one comparison of FJRW and twisted $5$-spin invariants]\label{thm:comparison}
At $\epsilon=\infty$, we have
\[
\sum_{n>1}\frac{1}{n!}\left\langle\left(\tau\phi_{1} \right)^n \right\rangle^{\w,\infty}_{1,n}=\sum_{n>1}\frac{1}{n!}\left\langle\left(\tau\phi_{1} \right)^n \right\rangle^{\lambda,\infty}_{1,n}+ \log(I_0(t))  \left( \langle \phi_0\psi_1 \rangle^{\w,\infty}_{1,1}- \langle \phi_0\psi_1 \rangle^{\lambda,\infty}_{1,1}\right) .
\]
\end{theorem}

\subsubsection{Part two: Explicit computations}

Having obtained a comparison between the genus-one FJRW and twisted $5$-spin invariants, our task is then to compute the twisted invariants explicitly. By applying the Givental--Teleman formula for semisimple CohFTs \cite{GiventalHamiltonians,Teleman} and computing the genus-zero data explicitly, we obtain the following formula.

\begin{theorem}[Computation of genus-one twisted $5$-spin invariants]\label{thm:twisted} The genus-one twisted $5$-spin invariants are given by
\[
\sum_{n>1}\frac{1}{n!}\left\langle\left(\tau\phi_{1} \right)^n \right\rangle^{\lambda,\infty}_{1,n} = \log \left( I_0(t)^{\frac{5}{24}-2}(1-({t}/{5})^{5}  )^{-\frac{1}{12}} \tau'(t)^{-\frac{1}{2}} \right) .
\]
\end{theorem}

In order to use Theorem \ref{thm:comparison} to obtain a formula for the genus-one FJRW invariants, we require the following simple computation.

\begin{lemma}\label{lem:onepoint}
The genus-one one-point invariants are given by
\[
 \langle \phi_0\psi_1\rangle^{\w,\infty}_{1,1} = -\frac{200}{24},\quad   \langle \phi_0\psi_1\rangle^{\lambda,\infty}_{1,1}= \frac{5}{24} .
\]
\end{lemma}

Combining Theorems \ref{thm:comparison} and \ref{thm:twisted} with Lemma \ref{lem:onepoint}, we obtain the following mirror formula.

\begin{theorem}[Genus-one FJRW mirror theorem]\label{thm:mirrortheorem}
We have the following explicit expression for the genus-one FJRW invariants:
\[
\sum_{n>1}\frac{1}{n!}\left\langle\left(\tau\phi_{1} \right)^n \right\rangle^{\w,\infty}_{1,n}= \log \left(( I_0(t))^{-\frac{31}{3}}(1-({t}/{5})^{5}  )^{-\frac{1}{12}} \tau'(t)^{-\frac{1}{2}} \right) .
\]
\end{theorem}

\begin{remark}
While the the wall-crossing result in Theorem \ref{thm:wc} is conceptually appealing and of independent interest, we note that the $\epsilon=0$ theory is not essential to the proof of the main result of this paper, which is Theorem \ref{thm:mirrortheorem}. Rather, the wall-crossing formula can simply be viewed as a convenient way to package the combinatorics of certain power series of rational tails that appear in the localization computations. 

In addition, while the twisted invariants are a necessary part of our proof, we expect that there is a way to circumvent the use of the twisted invariants and derive Theorem \ref{thm:mirrortheorem}  directly from the localization relations obtained from the moduli spaces of dual extended 5-spin curves. One benefit of our approach using twisted invariants is that the analysis of the twisted invariants that we carry out in this paper is also necessary for the arguments in the sequel \cite{GR}.
\end{remark}

\subsection{Plan of the paper}

In Section \ref{sec:master}, we study moduli spaces of dual-extended $5$-spin curves. These moduli spaces contain the moduli spaces of $5$-spin curves as special fixed loci of a natural $\bC^*$-action. By using localization, following Chang--Li--Li--Liu \cite{CLLL2}, we write down relations that determine all genus-one FJRW and twisted $5$-spin invariants. In Sections \ref{proof:loop} through \ref{proof:A}, we analyze the localization relations to prove Theorem \ref{thm:comparison}.

In Section \ref{sec6}, we review relevant notions of cohomological field theories and the genus-one formula for a generically semisimple CohFT, which we prove in Appendix A by applying Teleman's reconstruction theorem. In Section \ref{sec:comps}, we complete the proof of Theorem \ref{thm:mirrortheorem} by making the relevant genus-zero computations that appear in the genus-one formula.

\subsection{Further directions}

This work builds the potential for several new directions in regards to the Landau-Ginzburg/Calabi-Yau correspondence. Most notably, we use Theorem \ref{thm:mirrortheorem} in the sequel paper \cite{GR} to prove the genus-one version of the Chiodo--Ruan formulation of the LG/CY correspondence \cite{CR}. Namely, we prove that the quantization of the genus-zero symplectic transformation computed in \cite{CR} identifies the genus-one FJRW potential with the genus-one GW potential. This provides the first nontrivial evidence for the higher-genus conjecture.

The techniques developed in this paper for studying the localization relations on the master space can also be applied in higher genus. Of course, the higher-genus situation is more complicated for several reasons. We plan to devote future study to the higher-genus relations and what they say about higher-genus mirror formulas in both FJRW and GW theory.

\subsection{Acknowledgements}

The authors are greatly indebted to Yongbin Ruan for suggesting that they work together on this project, as well as for his invaluable guidance. The second author would also like to thank Emily Clader, Chiu-Chu Melissa Liu, and Mark Shoemaker for valuable discussions. The first author is partially supported by the NSFC grants 11431001 and 11501013. The second author has been supported by the NSF postdoctoral research fellowship DMS-1401873.

\section{Dual-extended $5$-spin curves and localization}\label{sec:master}

In this section, we review the definitions of auxiliary moduli spaces that contain the moduli spaces of $\epsilon$-stable $5$-spin curves (with P-fields) as special fixed loci of a natural $\bC^*$-action. These auxiliary moduli spaces are special cases of the so-called ``Master Space", which was introduced independently by Fan--Jarvis--Ruan \cite{FJRnew} and Chang--Li--Li--Liu \cite{CLLL1}. We describe two `virtual fundamental classes' on these moduli spaces; one recovering the Witten class on the special fixed loci of $5$-spin curves, and the other recovering the twisted virtual class. We also describe the virtual localization formula for these auxiliary moduli spaces, following Chang--Li--Li--Liu \cite{CLLL2}, and we outline how the structure of the localization formula leads to a proof of Theorem \ref{thm:wc}.

\subsection{Target geometry}

The moduli spaces that we study in this paper are special cases of those underlying the \emph{gauged linear sigma model} (GLSM). Developed in full generality and detail by Fan--Jarvis--Ruan \cite{FJRnew}, the GLSM is a generalization of Gromov-Witten theory for certain target spaces presented as GIT quotients. Here, we consider the GIT quotient
\[
X:=\left[\bC^7\GIT\bC^* \right]
\]
where $\bC^*$ acts on the coordinates by
\[
c\cdot(x_1,\dots,x_5,p,u)=(cx_1,\dots,cx_5,c^{-5}p,c^{-1}u).
\]
and GIT stability is chosen with respect to the negative linearization. It is not hard to see that, with these choices,
\[
X=\mathrm{Tot}\left(\cO_{\bP(5,1)}(-1/5)^{\oplus 5}\right)
\]
where $\bP(5,1)$ is the weighted projective line with homogeneous coordinates $(p,u)$. It will be useful in what follows to set some notation regarding the equivariant geometry of $X$. There are two primary settings that we investigate, corresponding to the two virtual classes. Since many of our main arguments in the two cases are parallel, we abuse notation by using the same symbols in each case.

\subsubsection{Case 1: $\T$-equivariant geometry}

We first consider an action of the torus $\T=\bC^*$ defined by
\[
t\cdot(x_1,\dots,x_5,p,u)=(x_1,\dots,x_5,p,tu).
\]
The $\T$-fixed loci of $X$ consist of the subspace
\[
X_\h:=\{(p,u)=(1,0)\}=[\bC^5/\mu_5]
\]
and the point
\[
X_\d:=\{(x_1,\dots,x_5,p,u)=(0,\dots,0,0,1)\}=\mathrm{pt}.
\]
The localized $\T$-equivariant Chen-Ruan cohomology $H_{\CR, \T}^*(X)$ has a fixed-point basis
\[
\varphi_0^\h,\varphi^\h_{1},\dots,\varphi^\h_{4},\varphi^\d
\]
where the superscript $\h$ corresponds to the cohomology of $X_\h$ and the superscript $\d$ corrresponds to the cohomology of $X_\d$. The subscripts index the twisted sectors of the inertia stack $\mathcal{I}X$, shifted by one; i.e. $\varphi_4$ is the untwisted sector while $\varphi_0$ is the first twisted sector, etc. 

Let $\alpha$ be the $\T$-equivariant parameter: $H_\T^*(\mathrm{pt})=\bQ[\alpha]$. We equip $H_{\CR, \T}^*(X)$ with a non-degenerate pairing $(-,-)$, defined by the following dual basis:\footnote{We warn the reader that this is not the equivariant Poincar\'e pairing.}
\begin{align*}
(\varphi^\h_{m\neq 4})^\vee&=\frac{1}{5}\varphi^\h_{3-m};\\
(\varphi^\h_4)^\vee&=\frac{1}{5}\alpha\varphi^\h_4;\\
(\varphi^\d)^\vee&=-\frac{1}{5}\alpha\varphi^\d.
\end{align*}
Define $(\eta^\bullet_m)^{-1}$ to be the $\varphi^\bullet_{3-m}$-coefficient of $(\varphi_m^\bullet)^\vee$. We denote the equivariant cohomology with this pairing by $H^{X,\w}:=H_{\CR, \T}^*(X)$ and its restriction to the fixed loci by $H^{\h,\w}$ and $H^{\d,\w}$.

\subsubsection{Case 2: $\S\times\T$-equivariant geometry}

Here, we consider the additional action of the torus $\S=(\bC^*)^{5}$ defined by 
\[
(s_1,\dots,s_5)\cdot(x_1,\dots,x_5,p,u)=(s_1x_1,\dots,s_5x_5,p,u).
\]
The $\S\times\T$-fixed loci of $X$ consist of two points
\[
X_\h:=\{x_1,\dots,x_5,p,u)=(0,\dots,0,1,0)\}=[\mathrm{pt}/\mu_5]
\]
and
\[
X_\d:=\{x_1,\dots,x_5,p,u)=(0,\dots,0,0,1)\}=\mathrm{pt}
\]
As in the $\T$-equivariant case, the localized $\S\times\T$-equivariant Chen-Ruan cohomology $H_{\CR, \S\times\T}^*(X)$ has a fixed-point basis
\[
\varphi_0^\h,\varphi^\h_{1},\dots,\varphi^\h_{4},\varphi^\d_0.
\]
We equip $H_{\CR, \S\times\T}^*(X)$ with a non-degenerate pairing $(-,-)$ defined exactly as in the case of $H_{\CR, \T}^*(X)$.  We denote the equivariant cohomology with this pairing by $H^{X,\lambda}:=H_{\CR, \S\times\T}^*(X)$ and its restrictions to the fixed loci by $H^{\h,\lambda}$ and $H^{\d,\lambda}$.

\subsection{Moduli Spaces}

We now describe the GLSM moduli spaces associated to the GIT quotient $X=[\bC^7\GIT\bC^*]$.
\begin{definition}
A \emph{dual-extended $5$-spin curve (with five P-fields)} is a tuple $(C,q_1,\dots,q_n,L,\sigma)$ where
\begin{itemize}
\item $(C,q_1,\dots,q_n)$ is a quasi-stable orbifold curve with possible $\mu_5$-orbifold structure at the marks and nodes,
\item $L$ is a representable orbifold line bundle on $C$, and
\item $\sigma=(x_1,\dots,x_5,p,u)$ is a section:
\begin{equation}\label{eq:section}
\sigma\in\Gamma(L^{\oplus 5}\oplus L^{-5}\otimes\omega_{C,\log}\oplus L^{-1}).
\end{equation}
\end{itemize}
Two triples $(C,L,\sigma)$, $(C',L',\sigma')$ are \emph{equivalent} if there exist isomorphisms $f:C\rightarrow C'$ and $\phi:L\rightarrow f^*L'$ such that $\phi\circ\sigma=f^*\sigma'$. Let $\vec m:=(m_1,\dots,m_n)$ record the multiplicities of $L$ at the marked points (i.e. the $\mu_5$-representation $L|_{q_i}$ is multiplication by $\exp(2\pi\sqrt{-1} m_i/5)$ with $m_i\in\{1,\dots,5\}$). We impose the additional condition that $x_i(q_j)=0$ for all $i,j$ (this is automatic for all $j$ with $m_j\neq 5$).

For any $\epsilon>0$, we define a triple $(C,L,\sigma)$ to be \emph{$\epsilon$-stable} if $(L^{-5}\otimes\omega_{C,\log})^\epsilon\otimes\omega_{C,\log}$ is ample and the locus of base points in $q\in C$ where $(p(q),u(q))=(0,0)$ is finite, disjoint from the marked and singular points on $C$, and each base point $q\in C$ has bounded order of vanishing:
\[
\ord_q(p,u)\leq1/\epsilon.
\]
Let $\M_{g,\vec m}^{\omega,\epsilon}(X,d)$ denote the moduli space parametrizing $\epsilon$-stable dual-extended spin curves up to isomorphism, where
\[
d:=\deg(L^{-1})>0.
\]
We write $\M_{g,n}^{\omega,\epsilon}(X,d)$ for the disjoint union over all possible multiplicity vectors $\vec m$ of length $n$.
\end{definition}

\begin{remark}
The notation $\M_{g,n}^{\omega,\epsilon}(X,d)$ is reminiscent of the notation for moduli spaces of stable quasi-maps. In fact, the only difference between $\M_{g,n}^{\omega,\epsilon}(X,d)$ and the moduli space of $\epsilon$-stable quasi-maps to $X$ is the $\omega_{C,\log}$ appearing in the sixth factor of \eqref{eq:section}. The $\omega$ in the superscript of the notation is meant to denote this twist. Also in analogy with $\epsilon$-stable quasi-maps, there are natural evaluation maps:
\[
\ev_i: \M_{g,n}^{\omega,\epsilon}(X,d)\rightarrow X.
\]
The existence of the evaluation maps follows from the fact that $\omega_{\log}$ is trivial upon restricting to the marked points.
\end{remark}

\begin{remark}\label{rmk:5spin}
When $u=0$, the section $p$ is equivalent to an isomorphism
\[
\kappa:L^5\cong\omega_{C,\log}\otimes\cO(-D)
\]
where $D$ is the divisor of zeros of $p$. When $\epsilon=\infty$, then $D$ is empty and we recover $\M_{g,\vec m}^{1/5}$. On the other hand, when $\epsilon=0$ we recover $\M_{g,\vec m\mid|D|}^{1/5,\epsilon}$. Thus, we see the moduli spaces of $\epsilon$-stable $5$-spin curves naturally appearing inside $\M_{g,n}^{\omega,\epsilon}(X,d)$.
\end{remark}

In this paper, we only consider the extreme cases $\epsilon=0$ ($0\leq \epsilon\ll 1$) and $\epsilon=\infty$ ($\epsilon\gg 0$). The arguments we provide can easily be generalized to arbitrary $\epsilon$, though we leave the details to the interested reader. We have the following important result due to Fan--Jarvis--Ruan for $\epsilon=0$ and Chang--Li--Li--Liu for $\epsilon=\infty$.

\begin{theorem}[\cite{FJRnew,CLLL1}]
For $\epsilon=0,\infty$, the moduli spaces of dual extended $5$-spin curves $\M_{g,\vec m}^{\omega,\epsilon}(X,d)$ are separated, Deligne-Mumford stacks, locally of finite type.
\end{theorem}
\begin{proof}

When $\epsilon=0$, $\M_{g,\vec m}^{\omega,0}(X,d)$ is a special case of Example 4.2.23 in \cite{FJRnew}. Thus, the required properties of the moduli space follow from Theorems 6.2.3 and 6.3.1 in \cite{FJRnew}.

When $\epsilon=\infty$, $\M_{g,\vec m}^{\omega,\infty}(X,d)$ is obtained from the moduli space of mixed-spin P-fields of \cite{CLLL1} by specializing the degree parameters to $d_0=0$ and $d_\infty=d$. Thus, the required properties follow from Theorem 1.1 in \cite{CLLL1}.
\end{proof}

\subsection{Virtual Classes}

We next describe two virtual classes on $\M_{g,\vec m}^{\omega,\epsilon}(X,d)$ that specialize to the Witten class and the twisted virtual class on the $5$-spin loci described in Remark \ref{rmk:5spin}

\subsubsection{Case 1: The Witten class}

In order to define the Witten class on $\M_{g,\vec m}^{\omega,\epsilon}(X,d)$, one requires a cosection of the obstruction sheaf. The cosection is defined by way of the following additional input.

\begin{definition}
The \emph{quintic super-potential} is defined by
\[
W=p(x_1^5+\cdots+x_5^5).
\]
The \emph{degeneracy locus}, denoted $\CM_{g,\vec m}^{\omega,\epsilon}(X,d)\subset\M_{g,\vec m}^{\omega,\epsilon}(X,d)$, consists of all dual-extended $5$-spin curves such that $\sigma$ maps fiber-wise to the critical points of $W$:
\[
\mathrm{Crit}_W:=\{p=x_1^5+\dots+x_5^5=0\}\cup\{x_1=\dots=x_5=0\}.
\]
\end{definition}

The following important result is due to Fan--Jarvis--Ruan for $\epsilon=0$ and Chang--Li--Li--Liu for $\epsilon=\infty$.

\begin{theorem}[\cite{FJRnew,CLLL1}]
For $\epsilon=0,\infty$, and $d>0$, $\CM_{g,\vec m}^{\omega,\epsilon}(X,d)$ is proper.
\end{theorem}

\begin{proof}

This is a special case of Theorem 1.1.1 in \cite{FJRnew} when $\epsilon=0$ and Theorem 1.1 in \cite{CLLL1} when $\epsilon=\infty$.
\end{proof}

Properness of the critical locus allows one to construct a virtual cycle via the cosection localization technique of Kiem--Li \cite{KL}.
\begin{theorem}[\cite{KL}]
There exists a cosection localized virtual cycle
\begin{equation}\label{eq:virclass}
[\M_{g,\vec m}^{\omega,\epsilon}(X,d)]^{\w}\in H_{\mathrm{virdim}}(\CM_{g,\vec m}^{\omega,\epsilon}(X,d),\bQ)
\end{equation}
where
\begin{equation}\label{eq:virdim}
\mathrm{virdim}=1-g+d+n-\sum_i\frac{4}{5}m_i.
\end{equation}
\end{theorem}

\subsubsection{Case 2: The twisted class}

To define an equivariant virtual class on $\M_{g,n}^{\omega,\epsilon}(X,d)$, consider the $\S$-action on $\M_{g,n}^{\omega,\epsilon}(X,d)$ defined by post-composing the section $\sigma$ with the $\S$-action on $X$. The fixed locus $\M_{g,n}^{\omega,\epsilon}(X,d)^\S$ consists of all sections $\sigma$ where $x_i=0$ for all $i$. The virtual normal bundle of this fixed locus is the restriction of $R\pi_*\L(-\Sigma_5)^{\oplus 5}$ where $\L$ is the universal line bundle over the universal curve, $\Sigma_5$ is the divisor of marked points with $m_i=5$, and $\pi$ is the projection from the universal curve to $\M_{g,n}^{\omega,\epsilon}(X,d)$.

\begin{definition}
The \emph{$\S$-equivariant twisted virtual class} is defined by
\[
[\M_{g,\vec m}^{\omega,\epsilon}(X,d)]^\lambda:=[\M_{g,\vec m}^{\omega,\epsilon}(X,d)^\S]^\vir\cap e_\S^{-1}\left(R\pi_*\L(-\Sigma_5)^{\oplus 5} \right)
\]
where $[\M_{g,\vec m}^{\omega,\epsilon}(X,d)^\S]^\vir$ is the virtual cycle induced by the relative perfect obstruction theory on $\M_{g,\vec m}^{\omega,\epsilon}(X,d)^\S$ given by $R\pi_*(\L^{-5}\otimes\omega_{\pi,\log}\oplus \L^{-1})$. After inverting the equivariant parameters, the twisted virtual class is an equivariant homology class in
\[
H_{\mathrm{virdim}}^\S(\M_{g,\vec m}^{\omega,\epsilon}(X,d)^\S,\bQ)
\]
where $\mathrm{virdim}$ is defined in \eqref{eq:virdim}.
\end{definition}

\begin{remark}\label{rmk:specialization}
We will ultimately be interested in studying the specialization
\[
\lambda_i=\xi^i\lambda
\]
where $\xi:=\re^{2\pi\ri/5}$. We will point out below where this specialization is required.
\end{remark}

\subsection{$\T$-action and equivariant correlators}

The spaces $\M_{g,\vec m}^{\omega,\epsilon}(X,d)$ and $\M_{g,\vec m}^{\omega,\epsilon}(X,d)^\S$ admit a $\T:=\bC^*$ action by scaling the last coordinate of the section $\sigma$:
\[
t\cdot(x_1,\dots,x_5,p,u):=(x_1,\dots,x_5,p,tu).
\]
Since the obstruction theory and the cosection are equivariant with respect to the $\T$-action, there is an equivariant cosection localized virtual cycle
\[
[\M_{g,\vec m}^{\omega,\epsilon}(X,d)]^\w\in H_{\mathrm{virdim}}^{\T}(\CM_{g,\vec m}^{\omega,\epsilon}(X,d),\bQ).
\]
Similarly, there is a canonical lift of the $\S$-equivariant twisted virtual class to the $\T$-equivariant setting
\[
[\M_{g,\vec m}^{\omega,\epsilon}(X,d)]^\lambda\in H_{\mathrm{virdim}}^{\S\times\T}(\M_{g,\vec m}^{\omega,\epsilon}(X,d)^\S,\bQ).
\]
For $\star=\w$ or $\lambda$ and classes $\varphi_i$ in the $m_i$th twisted sector of $H^{X,\star}$, we define two types of $\T$-equivariant correlators
\begin{equation}\label{eq:equivariantcorrelators}
\left\langle \varphi_1\psi^{a_1}\cdots\varphi_n\psi^{a_n}\right\rangle_{g,n|d}^{\star,\epsilon}:=\frac{(-1)^{5-5g-5d-\sum_{i}m_i}}{5^{2g-2}}\int_{[\M_{g,\vec m}^{\omega,\epsilon}(X,d)]^\star}\ev_1^*(\varphi_1)\psi_1^{a_1}\cdots\ev_n^*(\varphi_n)\psi_n^{a_n}.
\end{equation}
By results of Graber--Pandharipande \cite{GP} in the case $\star=\lambda$ and Chang--Kiem--Li \cite{CKL} in the case $\star=\w$, these equivariant correlators can be computed by virtual localization. Virtual localization on the moduli spaces of dual-extended $5$-spin curves is the primary tool used in the proof of Theorem \ref{thm:wc}.

\subsection{Virtual localization}

We now describe explicitly how to compute the equivariant correlators \eqref{eq:equivariantcorrelators} by restricting to the $\T$-fixed loci.

The $\T$-fixed loci in either $\M_{g,\vec m}^{\omega,\epsilon}(X,d)$ or  $\M_{g,\vec m}^{\omega,\epsilon}(X,d)^\S$ can be encoded by decorated bipartite graphs $\Gamma$. We denote the vertices, edges, legs, and flags of such a graph $\Gamma$ by $V$, $E$, $L$, and $F$, respectively. The legs are labeled by the set $\{1,\dots,n\}$, and we let $V=\h\sqcup \d$ denote the bipartite decomposition of the vertices. The decorations and the corresponding fixed loci are described as follows.
\begin{itemize}
\item The type-$\h$ vertices $v$ correspond to maximal connected components $C_v$ of $C$ where $u=0$, and the type-$\d$ vertices correspond to maximal connected components where $p=0$;
\item The legs adjacent to $v$, denoted $L_v$, record which of the marked points lie on $C_v$;
\item The edges correspond to irreducible rational components $C_e$ on which $x_i=0$ for all $i$ and the isomorphism class of $(C_e,\sigma)$ is fixed by the $\T$-action, i.e. we can write
\[
\sigma=(0,0,0,0,0,x^a,y^b)
\]
for a choice of homogeneous coordinates $[x,y]$ on $C_e$;
\item Each vertex $v$ is labeled with a genus $g_v$ and a degree $d_v$ recording the genus of $C_v$ and the degree of $L^{-1}$ restricted to $C_v$;
\item Each edge $e$ is labeled with a degree $d_e$ encoding the degree of $L^{-1}$ restricted to $C_e$.
\end{itemize}
The decorations must satisfy some constraints. In particular, we have the following:
\begin{itemize}
\item Genus constraint: $\sum_v g_v+b_1(\Gamma)=g$;
\item Degree constraint: $\sum_{v\in V} d_v+\sum_{e\in E} d_e=d$;
\item Integrality condition: For all vertices $v\in \h$ with adjacent edges $E_v$ and adjacent legs $L_v$:
\[
d_v+\sum_{e\in E_v} d_e+\sum_{i\in L_v}\frac{m_i}{5} \in\Z.
\]
\end{itemize}
In addition, when $\epsilon=0$, the stability condition disallows vertices $v$ with $2g_v-2+\val(v)\leq -1$.

Let $\Lambda_\epsilon$ denote the set of decorated graphs encoding such $\T$-fixed loci. For $\star=\w$ or $\lambda$, the virtual localization formula computes the equivariant correlator \eqref{eq:equivariantcorrelators} as a graph sum of the form
\begin{align}\label{eq:graphsum1}
\nonumber\sum_{\Gamma\in\Lambda_\epsilon} \frac{1}{|\Aut(\Gamma)|} &\prod_{v\in \h}\Contr_\epsilon^{\h,\star}(v)\prod_{v\in \d} \Contr_\epsilon^{\d,\star}(v)\\
&\cdot\prod_{e\in E}\Contr^\star(e)\prod_{(e,v)\in F \atop v\in\h}\Contr^{\h}(e,v)\prod_{(e,v)\in F \atop v\in\d}\Contr^{\d}(e,v).
\end{align}

Below, we collect the contributions of the vertices, edges, and flags to the graph sum \eqref{eq:graphsum1}. First, we introduce additional notation for the vertex contributions. We define the following correlators for the type-$\h$ vertices:
\[
\left\langle \phi_{m_1-1}\psi^{a_1}\cdots\phi_{m_n-1}\psi^{a_n} \right\rangle_{g,n|d}^{\h,\star,\epsilon}:=\frac{(-1)^{5-5g-5d-\sum_{i}m_i}}{5^{2g-2}}
\int_{\left[\M_{g,\vec m\mid 5d+2g-2+n}^{1/5,\epsilon}\right]^{\h,\star}}\frac{\psi_1^{a_1}\cdots\psi_n^{a_n}}{e_{\T}(R\pi_*\L^{-1})}.
\]
where
\[
\left[\M_{g,\vec m\mid 5d+2g-2+n}^{1/5,\epsilon}\right]^{\h,\star}\in \begin{cases}
H_{n-\sum m_i}\left(\M_{g,\vec m\mid 5d+2g-2+n}^{1/5,\epsilon}\right) & \star=\w\\
H_{n-\sum m_i}^{\S}\left(\M_{g,\vec m\mid 5d+2g-2+n}^{1/5,\epsilon}\right) & \star=\lambda
\end{cases}
\]
is the natural restriction of the virtual class to the vertex component $C_v$.\footnote{The fact that there is a virtual class on the component $C_v$ relies on the fact that $x_i|_{C_e}=0$, implying that $x_i=0$ on all preimages of nodes on $C_v$. Since the $\T$-action is trivial on the obstruction sheaves over $\h$-type vertices, the restriction of the virtual class lies in non-$\T$-equivariant homology.}

Similarly, we define type-$\d$ correlators:
\[
\left\langle \psi^{a_1}\cdots\psi^{a_n} \right\rangle_{g,n|d}^{\d,\star,\epsilon}:=\frac{(-1)^{5-5g-d}}{5^{2g-2}}\int_{\left[\M_{g,n\mid d}^\epsilon\right]^{\d,\star}}\frac{\psi_1^{a_1}\cdots\psi_n^{a_n}}{e_{\T}(R\pi_*(\L^{-5}\otimes\omega_{\log}))}
\]
where $\M_{g,n\mid d}^\epsilon$ is the Hassett space of weighted stable curves with $n$ usual marked points and $d$ indistinguishable weight-$\epsilon$ marked points forming a divisor $D$, $\L$ is the universal line bundle that restricts on fibers to $\cO(-D)$, and
\[
\left[\M_{g,n\mid d}^\epsilon\right]^{\d,\star}\in\begin{cases}
H_{2-2g-4n-4d}^\T\left(\M_{g,n\mid d}^\epsilon \right) & \star=\w\\
H^{\S\times\T}_{2-2g-4n-4d}\left(\M_{g,n\mid d}^\epsilon \right) & \star=\lambda
\end{cases}
\]
is the natural restriction of the equivariant virtual class to the vertex component $C_v$.

Define
\[
e_\star(-):=\begin{cases}
e_{\T}(-) & \star=\w,\\
e_{\S\times\T}(-) & \star=\lambda.
\end{cases}
\]

\begin{lemma}[\cite{CLLL2}]\label{lem:localization}
The localization contributions in \eqref{eq:graphsum1} can be computed by the following formulas:
\[
\Contr_\epsilon^{\h,\star}(v)=
Q^{d_v}\left\langle\prod_{e\in E_v}\frac{\phi_{5d_e-1}}{-\frac{\alpha}{d_e}-\psi}\right\rangle_{g_v,\val(v)|d_v}^{\h,\star,\epsilon};
\]
\[
\Contr_\epsilon^{\d,\star}(v)=Q^{d_v}\left\langle\prod_{e\in E_v}\frac{1}{\frac{\alpha}{d_e}-\psi} \right\rangle_{g_v,\val(v)|d_v}^{\d,\star,\epsilon};
\]
\[
\Contr^\star(e)=5^2\frac{(-1)^{h^0(L_{C_e}^{\oplus 5})-h^1(L_{C_e}^{\oplus 5})}Q^{d_e}}{|\Aut(C_e,\sigma)|}\frac{e_\star(H^1(C_e,L^{\oplus 5}))}{e_\star(H^0(C_e,L^{-5}\otimes\omega_{\log}\oplus L^{-1})^{\mov})};
\]
and
\[
\Contr^{\bullet}(e,v)=(\eta^\bullet_{5\langle d_e \rangle})^{-1}.
\]
The edge terms can be computed more explicitly. Let $v\in \h$ and $v'\in \d$ be the vertices adjacent to the edge $e$. When $2g_v-2+\val(v)>-1$ and $2g_{v'}-2+\val(v')>-1$, then
\[
\Contr^\lambda(e)=5\frac{Q^{d_e}}{d_e}\frac{\prod_{0<k<d_e \atop \langle k \rangle = \langle d_e \rangle}\left(k\frac{\alpha}{d_e}+\lambda_1\right)\cdots\left(k\frac{\alpha}{d_e}+\lambda_5\right)}{ (5d_e)!\left(\frac{\alpha}{d_e}\right)^{5d_e}\prod_{0\leq k <d_e \atop \langle k \rangle = \langle d_e \rangle}\left(k\frac{\alpha}{d_e}-\alpha \right)}.
\]
When $2g_v-2+\val(v)>-1$ and $2g_{v'}-2+\val(v')=-1$, then
\[
\Contr^\lambda(e)=5\frac{Q^{d_e}}{d_e}\frac{\prod_{0<k<d_e \atop \langle k \rangle = \langle d_e \rangle}\left(k\frac{\alpha}{d_e}+\lambda_1\right)\cdots\left(k\frac{\alpha}{d_e}+\lambda_5\right)}{ (5d_e-1)!\left(\frac{\alpha}{d_e}\right)^{5d_e-1}\prod_{0\leq k <d_e \atop \langle k \rangle = \langle d_e \rangle}\left(k\frac{\alpha}{d_e}-\alpha \right)}.
\]
When $2g_v-2+\val(v)=-1$ and $2g_{v'}-2+\val(v')>-1$, then $d_e\in\Z$ and
\[
\Contr^\lambda(e)=5\frac{ Q^{d_e}}{5d_e-1}\frac{\prod_{0<k<d_e-1/5 \atop \langle k \rangle = \langle d_e-1/5 \rangle}\left(k\frac{\alpha}{d_e-1/5}+\lambda_1\right)\cdots\left(k\frac{\alpha}{d_e-1/5}+\lambda_5\right)}{ (5d_e-1)!\left(\frac{\alpha}{d_e-1/5}\right)^{5d_e-1}\prod_{-1/5\leq k <d_e-1/5 \atop \langle k \rangle = \langle d_e-1/5 \rangle}\left(k\frac{\alpha}{d_e-1/5}-\alpha \right)}.
\]
The latter two cases only occur when $\epsilon=\infty$. In all cases, we have
\[
\Contr^\w(e)=\Contr^\lambda(e)_{\lambda_i=0}.
\]
As usual, we require special conventions for the unstable vertices $v$ where $2g_v-2+\val(v)\leq 0$. If $2g_v-2+\val(v)=-1$ with adjacent edge $e$ (this only occurs when $\epsilon=\infty$), we set
\[
\Contr^{\h,\star}_\epsilon(v)=-\frac{5\alpha}{5d_e-1}\Contr^{\h}(e,v)^{-1};
\]
and
\[
\Contr^{\d,\star}_\epsilon(v)=\frac{\alpha}{d_e}\Contr^{\d}(e,v)^{-1}.
\]
If $2g_v-2+\val(v)=0$ with adjacent edges $e_1$ and $e_2$ of degrees $d_1$ and $d_2$, then we set
\[
\Contr^{\h,\star}_\epsilon(v)=\left(-\frac{\alpha}{d_1}-\frac{\alpha}{d_2}\right)^{-1}\Contr^{\h}(e_1,v)^{-1};
\]
\[
\Contr^{\d,\star}_\epsilon(v)=\left(\frac{\alpha}{d_1}+\frac{\alpha}{d_2}\right)^{-1}\Contr^{\d}(e_1,v)^{-1}.
\]
If $2g_v-2+\val(v)=0$ with one adjacent leg and one adjacent edge, then $\Contr^{\bullet,\star}_\epsilon(v)=\Contr^{\bullet}(e,v)^{-1}$.

\end{lemma}

\begin{remark}
For $\star=\w$, these localization contributions are a special case of results of Chang--Li--Li--Liu in \cite{CLLL2}, Section 4. The computations in the $\star=\lambda$ case are similar.
\end{remark}

\subsection{Outline of the proof of Theorem \ref{thm:wc}}

Consider the genus-one dual-extended $5$-spin moduli spaces $\M_{1}^{\omega,\epsilon}(X,d)$ where $n=0$ and $d>0$. In this case, $\mathrm{virdim}=d>0$, and we have the following vanishing:
\begin{equation}\label{eq:vanishing}
0=\sum_{d>0}Q^d\langle - \rangle_{1,0|d}^{\w,\epsilon}=\left[\lambda^0\right]\sum_{d>0}Q^d\langle - \rangle_{1,0|d}^{\lambda,\epsilon}.
\end{equation}
Computing the correlators in \eqref{eq:vanishing} by virtual localization, the vanishing provides relations among the fixed loci. These relations are the key to proving Theorem \ref{thm:wc}.

More specifically, since the $\epsilon=0$ stability condition disallows rational tails, there are only three types of graphs that appear in $\Lambda_0$:
\begin{enumerate}
\item[$\Lambda_0^\h$:] Graphs consisting of a single vertex $v\in\h$ labeled with $g_v=1$ and $d_v=d$;
\item[$\Lambda_0^\d$:] Graphs consisting of a single vertex $v\in\d$ labeled with $g_v=1$ and $d_v=d$;
\item[$\Lambda_0^{\s}$:] Bivalent loops with each vertex $v$ labeled with $g_v=0$.
\end{enumerate}

For any graph $\Gamma\in\Lambda_\infty$, there is a unique graph $\Gamma_0\in\Lambda_0$ that is obtained from $\Gamma$ by `contracting the tails'. More specifically, to obtain $\Gamma_0$ from $\Gamma$, we carry out the following procedure.
\begin{enumerate}[1.]
\item Identify the unique genus-one vertex or the unique loop in $\Gamma$, call this subgraph $\Gamma_1$. As undecorated graphs, we set $\Gamma_0=\Gamma_1$.
\item For each vertex $v_0\in\Gamma_0$ and corresponding vertex $v_1\in\Gamma_1$, set $g_{v_0}:=g_{v_1}$.
\item For each edge $e_0\in\Gamma_0$ and corresponding edge $e_1\in\Gamma_1$, set $d_{e_0}:=d_{e_1}$.
\item For each vertex $v_0\in\Gamma_0$ and corresponding vertex $v_1\in\Gamma_1$, set
\[
d_{v_0}:=d_{v_1}+\deg(\Gamma_{v_1})
\]
where $\Gamma_{v_1}$ is the (possibly disconnected) graph consisting of the components of $\Gamma\setminus\Gamma_1$ that attach to $\Gamma_1$ at $v_1$, and $\deg(\Gamma_{v_1})$ is the sum of the degrees on all edges and vertices of $\Gamma_{v_1}$.
\end{enumerate}
To indicate the relationship between $\Gamma$ and $\Gamma_0$, we write $\Gamma\rightarrow\Gamma_0$.

From \eqref{eq:vanishing}, we obtain relations:
\begin{equation}\label{eq:vanish1}
0=\sum_{\Gamma\in\Lambda_\infty \atop \Gamma\rightarrow \Gamma_0\in\Lambda_0^\h}\Contr_\infty^\star(\Gamma)+\sum_{\Gamma\in\Lambda_\infty \atop \Gamma\rightarrow\Gamma_0\in\Lambda_0^\d}\Contr_\infty^\star(\Gamma)+\sum_{\Gamma\in\Lambda_\infty \atop \Gamma\rightarrow\Gamma_0\in\Lambda_0^\s}\Contr_\infty^\star(\Gamma)
\end{equation}
and
\begin{equation}\label{eq:vanish2}
0=\sum_{\Gamma\in\Lambda_0^\h}\Contr_0^\star(\Gamma)+\sum_{\Gamma\in\Lambda_0^\d}\Contr_0^\star(\Gamma)+\sum_{\Gamma\in\Lambda_0^\s}\Contr_0^\star(\Gamma)
\end{equation}
In the $\star=\lambda$ case, we are implicitly restricting to the $\lambda^0$ coefficient. Define the graph sum differences
\[
\Delta^\bullet:=\sum_{\Gamma\in\Lambda_\infty \atop \Gamma\rightarrow\Gamma_0\in\Lambda_0^\bullet}\Contr_\infty(\Gamma)-\sum_{\Gamma\in\Lambda_0^\bullet}\Contr_0(\Gamma).
\]
By \eqref{eq:vanish1} and \eqref{eq:vanish2}, we have $\Delta^\h=-\Delta^\d-\Delta^\s$. We prove the following two propositions in Sections \ref{proof:loop} and \ref{proof:B}, respectively.

\begin{proposition}\label{prop:loop}
There is an $\epsilon=0/\infty$ correspondence of contributions from loop-type graphs: $\Delta^\s=0$.
\end{proposition}

\begin{proposition}\label{prop:B}
There is an $\epsilon=0/\infty$ correspondence of contributions from graphs with a type-$\d$ vertex of genus one: $\Delta^\d=0$.
\end{proposition}

Propositions \ref{prop:loop} and \ref{prop:B} imply that there is an $\epsilon=0/\infty$ correspondence of contributions from graphs with a type-$\h$ vertex of genus one: $\Delta^\h=0$. Since the type-$\h$ vertices encode $5$-spin correlators, this is very close to the statement of Theorem \ref{thm:wc}. The final step in our proof of Theorem \ref{thm:wc}, proved in Section \ref{proof:A} by manipulating generating series, draws out the precise connection between the two statements.

\begin{proposition}\label{prop:A}
The correspondence $\Delta^\h=0$ implies Theorem \ref{thm:wc}.
\end{proposition}

\section{Proof of Proposition \ref{prop:loop}}\label{proof:loop}

In this section, we study equivariant intersection numbers on the genus-zero moduli spaces of dual-extended $5$-spin curves $\M_{0,n}^{\omega,\epsilon}(X,d)$. We prove various $\epsilon=0/\infty$ comparisons between genus-zero correlators, and we use these to deduce the loop-type correspondence of Proposition \ref{prop:loop}.

\subsection{Genus-zero generating series}

Our goal in this section is to study the genus-zero equivariant correlators defined in \eqref{eq:equivariantcorrelators}. Our computations require us to define formal generating series of these correlators, but first we require the relevant I-functions, which will suitably account for the unstable terms in the generating series.

\begin{definition}
The \emph{type-$(\h,\lambda)$ I-function} is defined by
\[
I^{\h,\lambda}(Q,z):=z\sum_{a\geq 0}\frac{Q^{\frac{a+1}{5}}}{z^{a}a!}\frac{\prod_{0< k<\frac{a+1}{5} \atop \langle k \rangle = \langle \frac{a+1}{5}\rangle}(kz+\lambda_1)\cdots(kz+\lambda_5)}{\prod_{0< k\leq\frac{a+1}{5} \atop \langle k \rangle = \langle \frac{a+1}{5}\rangle}(kz-\alpha)}\varphi^\h_{a}.
\]
The \emph{type-$(\d,\lambda)$ I-function} is defined by
\[
I^{\d,\lambda}(Q,z):=  -\frac{z\alpha}{5}\sum_{a\geq 0}\frac{Q^{a}}{z^aa!}\frac{\prod_{k=1}^{a-1}(kz+\alpha+\lambda_1)\cdots(kz+\alpha+\lambda_5)}{\prod_{k=0}^{5a-1}(kz+5\alpha)}\varphi^\d.
\]
The \emph{type-$(X,\lambda)$ I-function} is defined by
\[
I^{X,\lambda}(Q,z):=I^{\h,\lambda}(Q,z)+I^{\d,\lambda}(Q,z).
\]
For $\star=\w$, the type-$(\h,\w)$, type-$(\d,\w)$, and type-$(X,\w)$ I-functions are defined by taking $\lambda_i=0$ in the above expressions. We also define
\[
I^{\bullet,\star}(Q,z)^\epsilon=\begin{cases}
I^{\bullet,\star}(Q,z) & \epsilon=0\\
I^{\bullet,\star}(Q,z) \mod Q^{\frac{2}{5}} &\epsilon=\infty
\end{cases}
\]
where the latter case simply omits all but the $Q^0$ and $Q^{\frac{1}{5}}$ terms.
\end{definition}

We now define the big J-functions, which are the generating series of equivariant correlators that are required for our genus-zero comparisons.

\begin{definition}\label{def:J}
For $\bullet=\h$, $\d$, or $X$, and $\star=\w$ or $\lambda$, the corresponding \emph{big J-function} is defined by:
\[
J^{\bullet,\star}(t(z),Q,z)^\epsilon:=I^{\bullet,\star}(Q,z)^\epsilon+t(-z)+\sum_{n,d,\varphi}\frac{Q^d}{n!}\left\langle t(\psi)^n\frac{\varphi}{z-\psi} \right\rangle_{0,n|d}^{\bullet,\star,\epsilon}\varphi^\vee
\]
where
\[
t(z)\in\begin{cases}
H^{\bullet,\star}[[Q^\frac{1}{5}]][z] &\bullet=X\\
H^{\bullet,\star}[[Q^\frac{1}{5},z]] & \bullet=\h,\d
\end{cases}
\]
\end{definition}

\begin{remark}
The type-$\h$ and type-$\d$ correlators were already defined in Section \ref{sec:master}. To import those definitions into Definition \ref{def:J}, we are implicitly identifying $\varphi_m^\h$ with $\phi_m$ and $\varphi^\d$ with $1$.
\end{remark}

\begin{remark}
There is a subtle but very important convention regarding the expansions of the big $J$-functions. In particular, we expand $J^{X,\star}(t(z),Q,z)^\epsilon$ as a Laurent series in $z^{-1}$ over the base ring $H^{\bullet,\star}((Q^{\frac{1}{5}}))$, while we expand both $J^{\h,\star}(t(z),Q,z)^\epsilon$ and $J^{\d,\star}(t(z),Q,z)^\epsilon$ as Laurent series in $z$. %There are some subtle issues in dealing with the convergence of formal series, and this requires us to work in certain completions of spaces of formal series. These convergence issues have been carefully dealt with in the literature (see, for example, \cite{CCIT2}), and we do not address them explicitly here.
\end{remark}

\subsection{Genus-zero correspondences}

Here, we collect some $\epsilon=0/\infty$ correspondences for the genus-zero correlators. Since similar genus-zero correspondences have appeared recently in various places in the literature, for example \cite{Brown,CCIT2,RR,CladerRoss}, we keep the current discussion brief.

\begin{proposition}\label{prop:ab}
For $\bullet=\h$ or $\d$, expand $I^{\bullet,\star}(Q,z)$ as a power series in $z$ and write
\[
I^{\bullet,\star}(Q,z)=I^{\bullet,\star}(Q,z)_++I^{\bullet,\star}(Q,z)^\infty+\cO(z^{-1}).
\]
Then
\begin{equation}\label{eq:g=0}
J^{\bullet,\star}\bigg(t(z)+I^{\bullet,\star}(Q,-z)_+,Q,z\bigg)^\infty=J^{\bullet,\star}(t(z),Q,z)^0.
\end{equation}
\end{proposition}

\begin{proof}
This Proposition follows from arguments analogous to those developed by Ross--Ruan in \cite{RR}, which we now outline.

\subsubsection*{Step One.} We first modify \eqref{eq:g=0} in the case $\bullet=\h$ in order to introduce new formal variables that are more geometrically meaningful. Write $t(z)=\sum_{m,k}t_k^m\varphi_m^\h z^k$. We  make the following modification: multiply both sides of \eqref{eq:g=0} by $Q^{-1/5}$, then rewrite both sides in terms of $\q=Q^{1/5}$ and $\t_k^m=Q^{-1/5}t_k^m$. After these modifications, \eqref{eq:g=0} becomes the equality
\begin{equation}\label{eq:g=0'}
\hat J^{\h,\star}\left(\t(z)+\hat I^{\h,\star}(\q,-z)_+,\q,z\right)^\infty=\hat J^{\h,\star}(\t(z),\q,z)^0
\end{equation}
where the power of $\q$ in the $\hat J$-function tracks the number of weight-$\epsilon$ points on the underlying moduli space $\M_{0,n+1\mid 5d+n-1}^{1/5,\epsilon}$. When $\bullet=\d$, the power of $Q$ already tracks the number of light points, so we simply set $Q=\q$ and $\t_k=t_k$ to obtain the $\bullet = \d$ analog of \eqref{eq:g=0'}.

\subsubsection*{Step Two.} We have the following recursion.

\begin{lemma}\label{lem:rec}
With the pairing $(-,-)$ defined as in Section \ref{sec:master}, the Laurent series
\begin{equation}\label{eq:recursion}
\left(\frac{\partial}{\partial \q}\hat J^{\bullet,\star}(\t(z),\q,z)^\epsilon, \frac{\partial}{\partial \t_0^m}\hat J^{\bullet,\star}(\t(z),\q,-z)^\epsilon \right)
\end{equation}
is regular at $z=0$ for all $m$, for $\epsilon=0$ or $\infty$, and for $\bullet=\h$ or $\d$.
\end{lemma}

\begin{proof}[Proof of Lemma]

This Lemma is proved exactly as in the proof of Lemma 2.1 in Ross--Ruan \cite{RR}. More specifically, we begin by considering graph spaces, where we simply parametrize a rational component in the underlying modouli spaces. The graph spaces have a natural $\bC^*$ action by scaling the parametrized component. When we compute equivariant correlators on the graph spaces by virtual localization with respect to the $\bC^*$ action, the vertex terms in the localization formula encode the same correlators that are encoded by the big $J$-functions, leading to the localization expression \eqref{eq:recursion}, where $z$ is the equivariant parameter. In particular, the I-functions $\hat I^{\h,\star}$ and $\hat I^{\d,\star}$ can be computed as certain equivariant residues on the graph spaces. The fact that \eqref{eq:recursion} is regular at $z=0$ follows from the fact that the equivariant correlators are well-defined equivariant classes before localizing.
\end{proof}

\subsubsection*{Step Three.} We have the following characterization.

\begin{lemma}\label{lem:char}
Suppose that
\[
F(\t,\q,z)=F(\q,z)_++\t(z)+\underline{F}(\t,\q,z)
\]
is a Laurent series in $z$ over the base ring $H^{\bullet,\star}[[\q]]$ such that all terms of $\underline{F}(\t,\q,z)\in\cO(z^{-1})$ are at least quadratic in the variables $(\t,\q)$, and the Laurent series
\begin{equation}\label{eq:recursions}
\left(\frac{\partial}{\partial \q}F(\t,\q,z),\frac{\partial}{\partial \t_0^m}F(\t,\q,-z)\right)
\end{equation}
is regular at $z=0$ for all $m$. Then $F(\t,\q,z)$ is uniquely determined by $F(\q,z)_+$ and $F(\t,0,z)$.
\end{lemma}

\begin{proof}[Proof of Lemma]
The proof of this lemma follows that of Lemma 2.2 in Ross--Ruan \cite{RR}. More specifically, we write
\[
\underline{F}(\t,\q,z)=\sum_{j>0} f_{\vec n,d,k,m}\frac{\t^{\vec n}\q^d}{z^k}\varphi_m^\bullet
\]
where $\t^{\vec n}:=\prod (\t_l^{m})^{n_l^m}$. Our goal is to show that the coefficients $f_{\vec n,d,k,m}$ are completely determined from $F(\q,z)_+$, $f_{\vec n,0,k,m}$, and \eqref{eq:recursions}. To do so, we proceed by lexicographic induction on $(d,|\vec n|,k)$ where $|\vec n|:=\sum n_l^m$. Suppose $d>0$ and we know $f_{\vec n',d',k',m'}$ for all $(d',|\vec n'|,k')< (d,|\vec n|,k)$ and we want to compute $f_{\vec n,d,k,m}$. We consider the relation
\[
0=\left(\frac{\partial}{\partial \q}F(\t,\q,z),\frac{\partial}{\partial \t_0^m}F(\t,\q,-z)\right)\left[\frac{\t^{\vec n}q^{d-1}}{z^k} \right].
\]
This relation has an initial term equal to $d\cdot f_{\vec n,d,k,m}$ and all other terms are determined by the induction.
\end{proof}

\subsubsection*{Step Four.}

Our goal is to prove the equality
\[
\hat J^{\bullet,\star}\left(\t(z)+\hat I^{\bullet,\star}(\q,-z)_+,\q,z\right)^\infty=\hat J^{\bullet,\star}(\t(z),\q,z)^0.
\]
Using Lemma \ref{lem:rec}, both sides are easily seen to satisfy the hypotheses of Lemma \ref{lem:char}. In particular, the quadratic property follows from the fact that $\hat I^{\bullet,\star}(0,z)_+=0$. Moreover, it is easy to check that both sides agree modulo $z^{-1}$ and both sides agree when $\q=0$. Thus, Lemma \ref{lem:char} implies that the two sides are uniquely determined from the same initial data, and are thus equal, completing the proof of Proposition \ref{prop:ab}.
\end{proof}

\begin{proposition}\label{prop:de}
For $\star=\w$ or $\lambda$, we have
\[
J^{X,\star}(t(z),Q,z)^\infty=J^{X,\star}(t(z),Q,z)^0
\]
(without a change of variables).
\end{proposition}

\begin{proof}
This follows from Proposition \ref{prop:ab} using `cone-characterization' arguments analogous to those developed by Coates--Corti--Iritani--Tseng in \cite{CCIT2}, following ideas of Brown \cite{Brown} (see also Clader--Ross \cite{CladerRoss}, for a setting more analogous to the current one). As in the proof of Proposition \ref{prop:ab}, we content ourselves with outlining the main arguments.

\subsubsection*{Step One.} Let $F^\star=F^\star(t,Q,z)$ be a Laurent series in $z^{-1}$ over the base ring $H^{X,\star}((Q^{\frac{1}{5}}))$ such that $F^\star(0,0,z):=F^\star(0,Q,z)_{Q=0}$ exists and is regular at $z=0$. For $\bullet=\h$ or $\d$, let $F^{\bullet,\star}$ denote the restriction of $F$ to $H^{\bullet,\star}((Q^{\frac{1}{5}}))$ and let $F^{\bullet,\star}_{m}$ denote the coefficient of $\varphi_m^\bullet$. For $d\in\frac{1}{5}\N$, define the recursion coefficients by
\[
RC^\lambda(d):=\frac{1}{5d}\frac{\prod_{0<k<d \atop \langle k \rangle = \langle d \rangle}\left(k\frac{\alpha}{d}+\lambda_1\right)\cdots\left(k\frac{\alpha}{d}+\lambda_5\right)}{ (5d)!\left(\frac{\alpha}{d}\right)^{5d}\prod_{0\leq k <d \atop \langle k \rangle = \langle d \rangle}\left(k\frac{\alpha}{d}-\alpha \right)}
\]
and
\[
RC^\w(d):=RC^\lambda(d)_{\lambda_i=0},
\]
so that $RC^\star(d)=\Contr^\star(e)$ where $e$ is an edge in a localization graph of degree $d$.
We have the following characterization.

\begin{lemma}\label{lem:char2}
Suppose $F^\star$ satisfies the following three properties.
\begin{enumerate}
\item[(C1)] Each coefficient $F^{\bullet,\star}_m$ is a rational function of $z$. The coefficient $F^{\h,\star}_{m}$ has poles at $z=0$ and simple poles at $z=\frac{\alpha}{d}$ for all $d\in\frac{1}{5}\N$ with $5d=m \mod 5$. The coefficient $F^{\d,\star}$ has poles at $z=0$ and simple poles at $z=-\frac{\alpha}{d}$ for all $d\in\frac{1}{5}\N$.
\item[(C2)] The residues at the simple poles satisfy the following equations:
\[
\Res_{z=\frac{\alpha}{d}}F^{\h,\star}_{5d}=(\eta^\h_{5d})^{-1}Q^{d}RC^\star(d)\cdot F^{\d,\star}|_{z=\frac{\alpha}{d}}
\]
and
\[
\Res_{z=-\frac{\alpha}{d}}F^{\d,\star}=(\eta^\d)^{-1}Q^{d}RC^\star(d)\cdot F^{\h,\star}_{-5d}|_{z=-\frac{\alpha}{d}}
\]
\item[(C3)] With coefficients expanded as Laurent series in $z$, $F^{\bullet,\star}$ is of the form
\[
I^{\bullet,\star}(Q,z)^\infty+f(-z)+\sum_{n,d,\varphi}\frac{Q^d}{n!}\left\langle f(\psi)^n\frac{\varphi}{z-\psi} \right\rangle_{0,n|d}^{\bullet,\star,\infty}\varphi^\vee
\]
where $f(-z)$.
\end{enumerate}
Then, with coefficients expanded as Laurent series in $z^{-1}$, $F^\star$ is determined uniquely from the part with non-negative coefficients of $z$.\end{lemma}

\begin{proof}[Proof of Lemma]
By grouping the different poles, conditions (C1) and (C2) allow us to write
\begin{equation}\label{eq:rec1}
F^{\h,\star}_{m}=f^{\h,\star}_{m}(z)+\sum_{d\in\frac{1}{5}\N \atop \langle d \rangle = m}\frac{(\eta^\h_{5d})^{-1}Q^{d}RC(d)\cdot F^{\d,\star}|_{z=\frac{\alpha}{d}}}{z-\frac{\alpha}{d}}+\cO(z^{-1})
\end{equation}
and
\begin{equation}\label{eq:rec2}
F^{\d,\star}=f^{\d,\star}(z)+\sum_{d\in\frac{1}{5}\N}\frac{(\eta^\d)^{-1}Q^{d}RC(d)\cdot F^{\h,\star}_{-5d}|_{z=-\frac{\alpha}{d}}}{z+\frac{\alpha}{d}}+\cO(z^{-1})
\end{equation}
where $f^{\h,\star}_{m}(z)$ and $f^{\d,\star}(z)$ are polynomial in $z$ over the base ring. Expanding these expressions as Laurent series in $z$, property (C3), along with induction on the formal variables (using the fact that the second term in the right-hand sides of \eqref{eq:rec1} and \eqref{eq:rec2} always increases the power of $Q$), proves that the $\cO(z^{-1})$ part of $F^{\bullet,\star}$ is determined uniquely by $f^{\h,\star}_{m}(z)$ and $f^{\d,\star}(z)$, proving the claim.
\end{proof}

\subsubsection*{Step Two.} The next Lemma allows us to apply Lemma \ref{lem:char2} in our setting.
\begin{lemma}\label{lem:char3}
For $\epsilon= 0$ or $\infty$ and $\star=\w$ or $\lambda$, set
\[
F^{X,\star,\epsilon}:=J^{X,\star}(t(z),Q,z)^\epsilon.
\]
Then $F^{X,\star,\epsilon}$ satisfies conditions (C1) - (C3) of Lemma \ref{lem:char2}.
\end{lemma}

\begin{proof}[Proof of Lemma]
The restrictions of $F^{X,\star,\epsilon}$ to the fixed point basis are given by
\[
(I^\bullet(Q,z)^\epsilon)_{\bullet,m}+t^{\bullet,m}(-z)+\sum_{n,d}\frac{Q^d}{n!}\left\langle t(\psi)^n\;\frac{(\varphi_{\bullet,m})^\vee}{z-\psi} \right\rangle_{0,n|d}^{X,\star,\epsilon}.
\]
By definition, the two initial terms are rational in $z$ with the simple poles described in (C1). To verify the same for the final term, we apply virtual localization, as in Section \ref{sec:master}, to compute the correlators. There are two types of graphs:
\begin{enumerate}
\item[$\Gamma_1$:] Graphs where the last point is on a vertex of valence two;
\item[$\Gamma_2$:] Graphs where the last point is on a vertex of valence at least three.
\end{enumerate}
It is straightforward from the localization formulas in Section \ref{sec:master} that contributions from graphs in $\Gamma_1$ have the prescribed simple poles (notice that $\psi$ specializes to $\alpha/d$ on these fixed loci) while contributions from graphs in $\Gamma_2$ are polynomial in $z^{-1}$ (notice that $\psi$ is nilpotent in the type-$\h$ and type-$\d$ correlators). This proves (C1).

To verify condition (C2), consider all graphs in $\Gamma_1$. Notice that the recursive term $RC^\star(d)$ is exactly the contribution from the unique edge $e$ that is adjacent to the distinguished vertex supporting the last marked point. Therefore, the recursion in (C2) is equivalent to removing this edge from the graph. When $\epsilon=\infty$, the terms in $I^{X,\star}(Q,z)^\infty$ come into play when the opposite vertex of the removed edge has valence one (this can be checked using the three types of edge contributions described in Lemma \ref{lem:localization}). When $\epsilon=0$, each vertex must have valence at least two, and one needs to verify that $I^{X,\star}(Q,z)$ satisfies (C2), which is a straightforward computation.

To verify condition (C3), consider the restriction $F^{\bullet,\star,\epsilon}$. Define
\[
f^{\bullet,\star,\epsilon}(-z):=t^{\bullet}(-z)+\begin{cases}
\sum_{d\in\N}\frac{(\eta^\h_{5d})^{-1}Q^{d}RC(d)\cdot F^{\star,\epsilon}_{\d}|_{z=\frac{\alpha}{d}}}{z-\frac{\alpha}{d}}\varphi^\h_{5d} & \bullet=\h\\
\sum_{d\in\N}\frac{(\eta^\d)^{-1}Q^{d}RC(d)\cdot F^{\star,\epsilon}_{\h,-5d}|_{z=-\frac{\alpha}{d}}}{z+\frac{\alpha}{d}}\varphi^{\d} & \bullet=\d,
\end{cases}
\]
where, by (C2), the sums record the contribution from graphs in $\Gamma_1$. For any graph $\Gamma$, compute the localization contribution $\Contr(\Gamma)$ by first computing the contribution from each subgraph emanating from the distinguished vertex supporting the last marked point (each of which looks like a graph in $\Gamma_1$). Adding the contributions from each graph in this way, leads to the following equality.
\begin{equation}\label{eq:restrictseries}
F^{\bullet,\star,\epsilon}=I^\bullet(Q,z)^\epsilon+f^{\bullet,\star,\epsilon}(-z)+\sum_{n,d,\varphi}\frac{Q^d}{n!}\left\langle f^{\bullet,\star,\epsilon}(\psi)^n\frac{\varphi}{z-\psi} \right\rangle_{0,n|d}^{\bullet,\star,\epsilon}\varphi^\vee.
\end{equation}
For $\epsilon=\infty$, \eqref{eq:restrictseries} is precisely of the form required by (C3). For $\epsilon=0$, \eqref{eq:restrictseries} is of the form required by (C3) after applying Proposition \ref{prop:ab}.
\end{proof}

\subsubsection*{Step Three} From the definitions, it is apparent that, after expanding as Laurent series in $z^{-1}$,
\[
J^{X,\star}(t(z),Q,z)^0=J^{X,\star}(t(z),Q,z)^\infty\mod z^{-1}.
\]
By Lemmas \ref{lem:char2} and \ref{lem:char3}, this is enough to conclude that
\[
J^{X,\star}(t(z),Q,z)^0=J^{X,\star}(t(z),Q,z)^\infty,
\]
finishing the proof of Proposition \ref{prop:de}.
\end{proof}

\subsection{Tail Series}

In the process of passing from the graph $\Gamma\in\Lambda_\infty$ to $\Gamma_0\in\Lambda_0$, tails of rational curves with no marked points are removed.  We now define certain tail series that capture the contributions of those removed tails and we interpret the tail series explicitly in terms of $J^{X,\star}(0,Q,z)^\infty.$

Let $G_{d,m}^{\bullet,\w}$ denote the collection of $\T$-fixed loci in $\ev_{1}^*(\varphi_m^\bullet)\cap\M_{0,1}^{\omega,\infty}(X,d)$ where the unique marked point is on a vertex of valence two.  For each $G\in G_{d,m}^{\bullet,\w}$, let $\Contr_\infty^\w(G)$ denote the contribution of the fixed locus $G$ to the virtual localization formula. Similarly, let $G_{d,m}^{\bullet,\lambda}$ denote the collection of $\T$-fixed loci in $\ev_{1}^*(\varphi_m^\bullet)\cap\M_{0,1}^{\omega,\infty}(X,d)^\S$ where the unique marked point is on a vertex of valence two, and let $\Contr_\infty^\lambda(G)$ denote the contribution to the virtual localization formula.
\begin{definition}
For $\bullet=\h$ or $\d$ and $\star=\w$ or $\lambda$, the \emph{tail series} are defined by
\[
T^{\bullet,\star}(Q,z):=\sum_{d,m \atop G\in G_{d,m}^{\bullet,\star}} Q^d\frac{\Contr_\infty^\star(G)}{z-\psi}(\varphi_m^\bullet)^\vee.
\]
\end{definition}

\begin{remark}
The psi-class appearing in the definition of the tail series is purely equivariant, because the vertex supporting the marked point has no moduli.
\end{remark}

We have the following result relating the tail series to $J^{X,\star}(0,Q,z)^\infty$.

\begin{lemma}\label{lem:tail}
For $\bullet=\h$ or $\d$ and $\star=\w$ or $\lambda$, let $J^{X,\star}(0,Q,z)^\infty\big|_\bullet$ denote the restriction of $J^{X,\star}(0,Q,z)^\infty$ to $H^{\bullet,\star}$. Expanding as a Laurent series in $z$, write
\[
J^{X,\star}(0,Q,z)^\infty\big|_\bullet=J^{X,\star}(0,Q,z)^\infty\big|_{\bullet,+}+I^{\bullet,\star}(Q,z)^\infty+\cO(z^{-1}).
\]
Then
\[
T^{\bullet,\star}(Q,z)=J^{X,\star}(0,Q,z)^\infty\big|_{\bullet,+}
\]
\end{lemma}

\begin{proof}
Computing $J^{X,\star}(0,Q,z)^\infty\big|_\bullet$ by localization (ignoring the initial term $I^{\bullet,\star}(Q,z)^\infty$), there are two types of graphs that appear. The first type are those where the marked point is on a vertex of valence two, the second type are those where the marked point is on a vertex with valence at least three. The contributions of the former type, which are regular at $z=0$, are exactly those encoded by $T^{\bullet,\star}(Q,z)$, while the contributions of the latter type have poles at $z=0$.
\end{proof}

\subsection{Conclusion of Proposition \ref{prop:loop}}

In light of the localization computations of Section \ref{sec:master}, it is clear that Proposition \ref{prop:loop} is a consequence of the following identities:
\begin{equation}\label{prop:loop1}
\sum_{n>0}\left\langle \phi_{m_1}\psi^{a_1}\;\phi_{m_2}\psi^{a_2} \;\frac{T^{\h,\star}(Q,-\psi)^n}{n!}\right\rangle_{0,n+2}^{\h,\star,\infty}=\sum_{d>0} Q^d \left\langle \phi_{m_1}\psi^{a_1}\;\phi_{m_2}\psi^{a_2} \right\rangle_{0,2|d}^{\h,\star,0}
\end{equation}
and
\begin{equation}\label{prop:loop2}
\sum_{n>0}\left\langle \psi^{a_1}\;\psi^{a_2} \;\frac{T^{\d,\star}(Q,-\psi)^n}{n!}\right\rangle_{0,n+2}^{\d,\star,\infty}=\sum_{d>0} Q^d \left\langle \psi^{a_1}\;\psi^{a_2} \right\rangle_{0,2|d}^{\d,\star,0}
\end{equation}
Moreover, we compute
\begin{align}
\label{eq:tail0}T^{\bullet,\star}(Q,z)&=J^{X,\star}(0,Q,z)^\infty\big|_{\bullet,+}\\
&\label{eq:tail1}=J^{X,\star}(0,Q,z)^0\big|_{\bullet,+}\\
&\label{eq:tail2}=I^{\bullet,\star}(Q,z)_+
\end{align}
where \eqref{eq:tail0} follows from Lemma \ref{lem:tail}, \eqref{eq:tail1} follows from Proposition \ref{prop:de}, and \eqref{eq:tail2} follows from the definitions of $I^{\bullet,\star}(Q,z)_+$ and $J^{X,\star}(0,Q,z)^0\big|_{\bullet,+}$. Therefore, Equations \eqref{prop:loop1} and \eqref{prop:loop2} are special cases of Proposition \ref{prop:ab}, concluding the proof of Proposition \ref{prop:loop}.
\qed

\section{Proof of Proposition \ref{prop:B}}\label{proof:B}

In this section, we prove the $\epsilon=0/\infty$ comparison for graphs of type $\d$, which, after the discussion of tail series in Section \ref{proof:loop}, can be stated as the following equality:
\begin{equation}\label{prop:black1}
\sum_{n>0}\left\langle \frac{T^{\d,\star}(Q,-\psi)^n}{n!}\right\rangle_{1,n}^{\d,\star,\infty}=\sum_{d>0} Q^d \left\langle \;\;\right\rangle_{1,0|d}^{\d,\star,0}.
\end{equation}
The first step in proving \eqref{prop:black1} is to separate the genus-dependent part from the twisting factor in the type-$\d$ correlators. Once the genus-dependant part has been separated, the comparison can be reduced to a genus-zero statement using arguments first developed in the context of stable quotients by Marian--Oprea--Pandharipande \cite{MOP}. The genus-zero statement is proved by following arguments developed by Ciocan-Fontanine--Kim \cite{CFK} in the context of stable quasi-maps.

\subsection{Genus dependence of the type-$\d$ twisting factors}

The type-$\d$ correlators in \eqref{prop:black1} can be defined as intersection numbers on the moduli spaces $\M_{g,n, d}^\epsilon$ of $\epsilon$-stable curves with $n$ usual marked points $q_1,\dots,q_n$ and $d$ weight-$\epsilon$ points $y_1,\dots,y_d$. Let  $E=q_1+\dots+q_n$ and $D=y_1+\dots+y_d$ be the divisors of marked points, leading to universal divisors $\E$ and $\D$, and set $\L=\cO(-\D)$. The type-$\d$ correlators are obtained by intersecting psi-classes against
\begin{equation}\label{eq:twist}
\frac{1}{d!}e_{\star}^{-1}\left(R\pi_*\left(\L^{\oplus 5}\oplus\L^{-5}\otimes\omega_{\pi,\log} \right)\right).
\end{equation}
\begin{remark}
The $1/d!$ pre-factor occurs here because we are marking the weight-$\epsilon$ points in this discussion, whereas we considered them to be indistinguishable in the discussion of Section \ref{sec:master}. Notationally, we have
\[
\M_{g,n\mid d}^\epsilon=\M_{g,n, d}^\epsilon/S_d
\]
\end{remark}

The following lemma separates out the genus-dependent part of the twisting factor \eqref{eq:twist}.
\begin{lemma}\label{lem:separate}
The twisting factor separates into a genus-dependent part and a part that is local to the divisor $\D$:
\[
e_{\star}^{-1}\left(R\pi_*\left(\L^{\oplus 5}\oplus\L^{-5}\otimes\omega_{\pi,\log} \right)\right)=e_\star^{-1}\left(R\pi_*\left(\cO^{\oplus 5}\oplus\omega_\pi\right)\right)\frac{e_\star\left(R\pi_*\left(\cO|_{\D}^{\oplus 5}\right)\right)}{e_\star\left(R\pi_*\left(\cO(20\D)|_{\D}\right)\right)}.
\]
\end{lemma}

\begin{proof}
This follows from the two exact sequences
\[
0\rightarrow\cO(-\D)\rightarrow\cO\rightarrow \cO|_\D\rightarrow 0
\]
and
\[
0\rightarrow \cO\otimes\omega_\pi\rightarrow \cO(5\D+\E)\otimes\omega_\pi\rightarrow\left(\cO(5\D+\E)\otimes\omega_\pi\right)|_{5\D+\E}\rightarrow 0,
\]
along with the facts that $\D$ and $\E$ are disjoint, $\omega_\pi\otimes\cO(\E)|_{\E}=\cO$, and
\begin{align*}
\omega_\pi\otimes\cO(5\D)|_{5\D}&=\cO(4\D)\otimes\left(\omega_\pi\otimes\cO(\D)\right)|_{5\D}\\
&=\cO(4\D)|_{5\D}\\
&=\cO(20\D)|_{\D}.
\end{align*}
\end{proof}

\subsection{Contraction maps and cotangent calculus}

Corresponding to the graph contraction $\Gamma\rightarrow\Gamma_0\in\Gamma_0^\d$, there is a contraction map on moduli spaces:
\[
\rho_\Gamma:\M^\infty_{1,n}\rightarrow\M^0_{1,0,d=d_1+\dots+d_n}
\]
where the set $\{1,\dots,n\}$ enumerates the connected graphs of $\Gamma\setminus\Gamma_1$ and $d_i$ denotes the degree on the $i$th such graph. The map $\rho_\Gamma$ simply replaces the marked point $q_i$ with the $d_i$ weight-$0$ points that are indexed by the interval $[d_1+\dots+d_{i-1}+1,d_1+\dots+d_i]$, then stabilizes.

Since
\[
\rho_\Gamma^*\left(e_{\star}^{-1}\left(R\pi_*\left(\cO^{\oplus 5}\oplus\omega_\pi \right)\right)\right)=e_{\star}^{-1}\left(R\pi_*\left(\cO^{\oplus 5}\oplus\omega_\pi \right)\right),
\]
then Lemma \ref{lem:separate}, along with the projection formula, implies that \eqref{prop:black1} would follow from the equality
\begin{equation}\label{prop:black2}
\sum_{d_1,\dots,d_n \atop \sum_i d_i=d}\frac{1}{n!}(\rho_\Gamma)_*\left(\prod_iT^{\d,\star}(Q,-\psi)[Q^{d_i}]\right)=\frac{1}{d!}\frac{e_{\star}\left(R\pi_*\left(\cO|_\D^{\oplus 5}\right)\right)}{e_{\star}\left(R\pi_*\left(\cO(20\D)|_{\D}\right)\right)}
\end{equation}
as an equation in the equivariant cohomology ring
\[
\begin{cases}
H_{\T}^*(\M_{1,0,d}^0) & \star=\w\\
H_{\S\times\T}^*(\M_{1,0,d}^0) &\star=\lambda.
\end{cases}
\]

Both sides of \eqref{prop:black2} can be written as polynomials in psi-classes $\hat\psi_j$ and diagonal classes $D_J$, symmetric under the action of $S_d$. We denote these polynomials by $B_d(\hat\psi_j,D_J)^\infty$ and $B_d(\hat\psi_j,D_J)^0$. By the cotangent calculus of Marian--Oprea--Pandharipande \cite{MOP}, these polynomials can be written in a canonical form, which we denote by $B^C$, and the canonical forms are also symmetric under the action of $S_d$. Our goal is to show, not just the cohomological equality \eqref{prop:black2}, but the stronger equality
\begin{equation}\label{eq:poly}
B_d^C(\hat\psi_j,D_J)^\infty=B_d^C(\hat\psi_j,D_J)^0
\end{equation}
as abstract polynomials in the variables $\hat\psi_j,D_J$.

Marian--Oprea--Pandharipande argue that any equality of abstract symmetric polynomials in variables $\hat\psi_j$ and $D_J$ can be checked by computing certain intersections on \emph{genus-zero} moduli spaces $\M_{0,k|d}^0=\M_{0,k,d}/S_d$. To precisely define the intersection numbers we need to check, we follow the exposition in Ciocan-Fontanine--Kim \cite{CFK}. Fix $k\geq 3$, $d\geq 0$, and $1\leq l\leq k-2$. Let $\rho=(p_1,\dots,p_l)$ and $\tau=(t_1,\dots,t_l)$ be ordered partitions of $d$ and $k-2$, respectively, with $p_i,t_i\geq 0$. Define the \emph{chain-type stratum} $S(\tau,\rho)$ on $\M_{0,k|d}^0$ to be the closure of the locus parametrizing curves with $l$ irreducible components $R_1,\dots,R_l$ attached in a chain such that
\begin{itemize}
\item $R_1$ carries $t_1+1$ regular marked points and $p_1$ weight-$0$ marked points,
\item for $i=2,\dots,l-1$, $R_i$ carries $t_i$ regular marked points and $p_i$ weight-$0$ marked points, and
\item $R_l$ carries $t_l+1$ regular marked points and $p_l$ weight-$0$ marked points,
\end{itemize}
where the regular marked points are distributed in order from $R_1$ to $R_l$. The key lemma we need, which was proved by Marian--Oprea--Pandharipande, is the following.

\begin{lemma}[\cite{MOP} Section 7.6; \cite{CFK} Section 5.6]
The equality of abstract polynomials in \eqref{eq:poly} holds if and only if, for every chain-type stratum $S(\tau,\P)$ and every monomial $\mu(\psi_1,\dots,\psi_k)$, we have an equality after integrating:
\begin{equation}\label{eq:neceq}
\int_{S(\tau,\P)}\mu(\psi_1,\dots,\psi_k)B_d^C(\hat\psi_j,D_J)^\infty=\int_{S(\tau,\P)}\mu(\psi_1,\dots,\psi_k)B_d^C(\hat\psi_j,D_J)^0.
\end{equation}
\end{lemma}

To prove \eqref{eq:neceq} for all chain-type strata, and thus finish the proof of Proposition \ref{prop:B}, we proceed by induction on $l$. Back-tracking through the definitions and results of Section \ref{proof:loop}, we see that the base case $l=1$ is precisely encoded by the $\bullet=\d$ case of Proposition \ref{prop:ab}. When $l>1$, we simply break the chain-type strata at the first node and denote the two resulting chain-type strata by $S_1$ and $S_2$. Set $d_1=p_1$ and $d_2=d-d_1$. Then it is not hard to see that, for $\epsilon=0$ or $\infty$, we have splittings
\[
B_d^C(\hat\psi_j,D_J)^\epsilon=B_{d_1}^C(\hat\psi_j,D_J)^\epsilon B_{d_2}^C(\hat\psi_j,D_J)^\epsilon
\]
and
\[
\mu(\psi_1,\dots,\psi_k)=\mu_1(\psi_1,\dots,\psi_{t_1+1})\mu_2(\psi_{t_1+2},\dots,\psi_n),
\]
and we can write
\[
\int_{S(\tau,\P)}\mu B_d^C(\hat\psi_j,D_J)^\epsilon=\int_{S_1}\mu_1B_{d_1}^C(\hat\psi_j,D_J)^\epsilon\cdot\int_{S_2}\mu_2B_{d_2}^C(\hat\psi_j,D_J)^\epsilon.
\]
Therefore, the equality of the integrals can be reduced to an equality on chain-type strata with smaller $l$, completing the induction step.

This finishes the proof of Proposition \ref{prop:B}.
\qed

\section{Proof of Proposition \ref{prop:A}}\label{proof:A}

Propositions \ref{prop:loop} and \ref{prop:B} together imply the following correspondence between the type-$\h$ correlators:
\begin{equation}\label{prop:A1}
\sum_{n>0}\frac{Q^{-n/5}}{n!}\left\langle T^{\h,\star}(Q,-\psi)^n \right\rangle_{1,n|d=-n/5}^{\h,\star,\infty}=\sum_{d>0} Q^d \left\langle \;\;\right\rangle_{1,0|d}^{\h,\star,0},
\end{equation}
where, when $\star=\lambda$, we only consider the coefficient of $\lambda^0$ on each side. Now we manipulate both sides of \eqref{prop:A1} to show that it implies the statement of Theorem \ref{thm:wc}.

\subsection{Right-hand side}

The correlators on the right-hand side of \eqref{prop:A1} are defined by
\[
\int_{\left[\M_{1,0\mid 5d}^{1/5,0}\right]^\star} e_\T^{-1}(R\pi_*\L^{-1}).
\]
Since
\[
\dim \left[\M_{1,0\mid 5d}^{1/5,0}\right]^\star=0,
\]
the twisting factor can only contribute the purely equivariant leading term:
\[
e_\T^{-1}(R\pi_*\L^{-1})=(-\alpha)^{-d}+\dots
\]
where $+\dots$ denotes terms that are not purely equivariant. Therefore, setting $t=(-\alpha^{-1}Q)^{1/5}$, it follows from the definitions that the right-hand side of \eqref{prop:A1} simplifies:
\begin{equation}\label{eq:RHS}
\sum_{d>0} Q^d \left\langle \;\;\right\rangle_{1,0|d}^{\h,\star,0}=\sum_{d>0} t^d \left\langle -\right\rangle_{1,0|d}^{\star,0},
\end{equation}
which is equal to the final term in the right-hand side of Theorem \ref{thm:wc}.

\subsection{Left-hand side}

It is left to recover the rest of the terms in Theorem \ref{thm:wc} from the left-hand side of \eqref{prop:A1}. We simplify the left-hand side by making several observations. When $\star=\lambda$, we require the specialization $\lambda_i=\xi^i\lambda$ from Remark \ref{rmk:specialization}.

First, recall from \eqref{eq:tail0} - \eqref{eq:tail2} that
\[
T^{\h,\star}(Q,z)=I^{\h,\star}(Q,z)_+.
\]
For $\star=\lambda$ or $\w$, respectively, we compute directly from the definitions the $\varphi_{0}^\h$-coefficients of $I^{\h,\star}(Q,z)_+$:
\[
T^{\h,\star}(Q,z)[\varphi_{0}^\h]=\begin{cases}
\left[z\sum_{a>0 \atop 5|a}\frac{Q^{\frac{a+1}{5}}}{z^aa!}\frac{\prod_{0\leq k<\frac{a+1}{5} \atop \langle k \rangle = \langle \frac{a+1}{5}\rangle}((kz)^5+\lambda^5)}{\prod_{0< k\leq\frac{a+1}{5} \atop \langle k \rangle = \langle \frac{a+1}{5}\rangle}(kz-\alpha)}\right]_+=\cO(z)+\cO(\lambda^5)\\
\left[z\sum_{a>0 \atop 5|a}\frac{Q^{\frac{a+1}{5}}}{z^aa!}\frac{\prod_{0\leq k<\frac{a+1}{5} \atop \langle k \rangle = \langle \frac{a+1}{5}\rangle}(kz)^5}{\prod_{0< k\leq\frac{a+1}{5} \atop \langle k \rangle = \langle \frac{a+1}{5}\rangle}(kz-\alpha)}\right]_+=\cO(z).
\end{cases}
\]
%and
%\[
%T^{\h,\star}(Q,z)[\varphi_{4}^\h]=\begin{cases}
%\left[z\sum_{a>0 \atop 5|a+1}\frac{Q^{\frac{a+1}{5}}}{z^aa!}\frac{\prod_{0\leq k<\frac{a+1}{5} \atop \langle k \rangle = \langle \frac{a+1}{5}\rangle}((-kz)^5-\lambda^5)}{\prod_{0< k\leq\frac{a+1}{5} \atop \langle k \rangle = \langle \frac{a+1}{5}\rangle}(kz-\alpha)}\right]_+=\cO(\lambda^5)\\
%0
%\end{cases}
%\]
This means that every time we see $\varphi_{0}^\h$ in the correlator in the left-hand side of \eqref{prop:A1}, it appears either with a $\psi$ class or a factor of $\lambda^5$. Next, notice that
\[
\dim \left[\M_{1,\vec m}^{1/5^\infty}\right]^\star= 2n-\sum_{i} m_i.
\]
Since an appearance of a $\psi$ class takes up one dimension and the appearance of a factor of $\lambda^5$ effectively takes up five dimensions (recall, here, that we are only considering the $\lambda^0$ coefficient in the left-hand side of \eqref{prop:A1}), we see by purely dimensional reasons that the only possible insertions in the left-hand side of \eqref{prop:A1} appear as coefficients of $\varphi^\h_{0}\psi$ and $\varphi^\h_{1}$. We compute directly the coefficients of these insertions in the tail series:
\[
T^{\h,\star}(Q,-z)[\varphi^\h_{0}z\lambda^0]=-\alpha^{-1}Q^{1/5}I_0((-\alpha^{-1} Q)^{1/5})+\alpha^{-1}Q^{1/5}
\]
and
\[
T^{\h,\star}(Q,-z)[\varphi^\h_{1}]=(-\alpha)^{-4/5} Q^{1/5}I_1((-\alpha^{-1} Q)^{1/5}),
\]
where $I_0(t)$ and $I_1(t)$ were defined in the introduction.

Lastly, notice that when all insertions are of type $\varphi^\h_{0}\psi$ and $\varphi^\h_{1}$, then by the same dimension count above, the twisting factor can only contribute the purely equivariant term:
\[
e_\T^{-1}(R\pi_*\L^{-1})=\dots+(-\alpha)^{n_{1}+\frac{4}{5}n_{2}}+\dots
\]
where $n_m$ denotes the number of points of multiplicity $m$.

Putting these three observations together and setting $t=(-\alpha^{-1}Q)^{1/5}$ as before, the left-hand side of \eqref{prop:A1} simplifies to
\begin{align}\label{eq:LHS}
\nonumber\sum_{n>0}\frac{Q^{-n/5}}{n!}&\left\langle T^{\h,\star}(Q,-\psi)^n \right\rangle_{1,n,-n/5}^{\h,\star,\infty}\\
&=\sum_{n>0}\frac{1}{n!}\left\langle \left[\left(1-I_0(t)\right)\phi_{0}\psi+I_1(t)\phi_1\right]^n \right\rangle_{1,n}^{\star,\infty}.
\end{align}
Finally, we rewrite the right-hand side \eqref{eq:LHS} as follows:
\begin{align}
\nonumber\sum_{n>0}\frac{1}{n!}&\left\langle  \left[\left(1-I_0(t)\right)\phi_{0}\psi+I_1(t)\phi_1\right]^n \right\rangle_{1,n}^{\star,\infty}\\
&\label{dilaton2}=\sum_{n_1>0}\frac{1}{n_1!}\left\langle  \left[\left(1-I_0(t)\right)\phi_{0}\psi\right]^{n_1} \right\rangle_{1,n_1}^{\star,\infty}\\
&\label{dilaton3}\hspace{.2cm}+\sum_{n_1\geq 0 \atop n_2>0}\frac{1}{n_1!n_2!}\left\langle  \left[\left(1-I_0(t)\right)\phi_{0}\psi\right]^{n_1}\left[I_1(t)\phi_1 \right]^{n_2} \right\rangle_{1,n_1+n_2}^{\star,\infty}.
\end{align}
The FJRW and twisted correlators satisfy the dilaton equation:
\[
\left\langle\phi_{m_1}\psi^{a_1}\cdots\phi_{m_n}\psi^{a_n}\;\phi_{0}\psi \right\rangle_{g,n+1}^{\star,\infty}=(2g-2+n)\left\langle\phi_{m_1}\psi^{a_1}\cdots\phi_{m_n}\psi^{a_n} \right\rangle_{g,n}^{\star,\infty}
\]
whenever $2g-2+n>0$. Applying the dilaton equation, we see that the sum in \eqref{dilaton2} is equal to
\[
-\log(I_0(t))\int_{\left[\M_{1,(1/5)}^{1/5,\infty}\right]^\star}\psi_1.
\]
Also applying the dilaton equation to \eqref{dilaton3} and then using the identity
\[
\sum_{n_1\geq 0}{n_1+n_2-1 \choose n_1}(1-I_0(t))^{n_1}=\frac{1}{(1-(1-I_0(t)))^{n_2}}=\frac{1}{I_0(t)^{n_2}},
\]
the sum in \eqref{dilaton3} simplifies to
\[
\sum_{n>1}\frac{1}{n!}\left\langle\left(\frac{I_1(t)}{I_0(t)}\phi_1 \right)^n \right\rangle^{\infty,\star}_{1,n}.
\]
This completes the proof of Theorem \ref{thm:wc}.
\qed

\section{Twisted invariants and the genus-one formula} \label{sec6}

Now that we have completed the proof of the comparison formula in Theorem \ref{thm:comparison}, we now turn towards proving the explicit formula for the genus-one twisted invariants given in Theorem \ref{thm:twisted}. In this section, we recall a formula that computes the genus-one twisted correlators purely in terms of certain genus-zero data. This formula is originally due to Dubrovin--Zhang \cite{DZ} and Givental \cite{Giv}. In Section \ref{sec:comps}, we compute the relevant genus-zero data explicitly in order to deduce Theorem \ref{thm:twisted} from the genus-one formula.

\subsection{The CohFT and Frobenius manifold}

In order to state the genus-one formula, we first need to introduce the twisted invariants and the requisite genus-zero data. Recall from \cite{CR} that twisted $5$-spin invariants can be extended to the untwisted sector by the following formula.
\begin{align*}
\langle\phi_{m_1-1}\psi^{a_1}\cdots\phi_{m_n-1}\psi^{a_n} \rangle_{g,n}^\lambda&:=\frac{(-1)^{3-3g+n-\sum m_i}}{5^{2g-2}}\int_{\left[\M_{g,\vec m}^{1/5}\right]^\lambda}\psi_1^{a_1}\cdots\psi_n^{a_n}\\
&=\frac{1}{5^{2g-2}}\int_{\left[\M_{g,\vec m}^{1/5}\right]}\psi_1^{a_1}\cdots\psi_n^{a_n}c_\lambda(\vec m)
\end{align*}
where
\[
c_\lambda(\vec m) :=  \re^{\sum_{i=1}^5\sum_{k\geq 0}s_{k,i}\ch_k(-[R\pi_*\L(-\Sigma_5)])}
\]
with
\[
s_{k,i}:=\begin{cases}
-\ln(\lambda_i) & k=0,\\
\frac{(k-1)!}{\lambda_i^k} & k>0,
\end{cases}
\]
and $\Sigma_5$ is the universal divisor of marked points with trivial twisting. We further specialize the formal parameters $\lambda_i$ by setting $\lambda_i=\xi^i\lambda$ with $\xi:=\re^{2\pi\ri/5}$. Notice that this choice of the parameters $\lambda_i$ induces the vanishing $s_k=0$ unless $5|k$.  By orbifold Riemann-Roch,
\[
-\rk(R^0\pi_*\L(-\Sigma_5)) + \rk(R^1\pi_*\L(-\Sigma_5)) =\frac{3g-3-n+\sum m_i}{5},
\]
where we always take $m_i\in\{1,\dots,5\}$. It follows that the characteristic class $c_\lambda$ can be written in the following form:
\[
c_\lambda(\vec m) = \cdots +   \lambda^{-5} c_{d+5}+ c_d + \lambda^5 c_{d-5} +\lambda^{10} c_{d-10}+ \cdots
\]
where $d =3g-3-n+\sum m_i$. When $g=0$, then
 \[
 -[R\pi_*\L(-\Sigma_5)]=R^1\pi_*\L(-\Sigma_5)
 \]
is a vector bundle, and we have
\begin{equation} \label{eq:equichern}
c_\lambda(\vec m) =  c_d +\lambda^5 c_{d-5} + \cdots.
\end{equation}

For a formal parameter $\tau$ (which we later take to be the mirror map $\tau(t)$), we define the \emph{shifted twisted invariants in the small phase space} by
\[
\left\langle\left\langle\phi_{m_1-1}\psi^{a_1}\cdots\phi_{m_n-1}\psi^{a_n} \right\rangle\right\rangle^\lambda_{g,n}:=\sum_{k\geq 0}\frac{\tau^k}{k!}
\left\langle\phi_{m_1-1}\psi^{a_1}\cdots\phi_{m_n-1}\psi^{a_n}\; \phi_1\cdots \phi_1 \right\rangle^\lambda_{g,n+k} .
\]

We now describe the CohFT and the underlying Frobenius manifold corresponding to the shifted twisted invariants. For the basic definitions of CohFTs, we suggest the discussion in Pandharipande--Pixton--Zvonkine \cite{PPZ}, while for Frobenius manifolds we direct the reader to Givental \cite{Giv} and Lee-Pandharipande \cite{LP}.

The CohFT associated to the shifted twisted theory is based on the vector space generated by $\phi_0,\dots,\phi_4$, with unit $\phi_0$, and is defined by
\[
\Omega^\tau_{g,n}(\phi_{m_1-1}\otimes\cdots \otimes \phi_{m_n-1}) :=\sum_{k\geq 0}\frac{\tau^k}{k!} (p_k)_*(c_\lambda(\vec m,2^k)),
\]
where $p_k: \M_{g,(\vec m,2^k) }^{1/5} \rightarrow \M_{g,\vec m }^{1/5} \rightarrow \M_{g,n}$ is the forgetful map that forgets the last $k$ marked points, the line bundle, and the orbifold structure. The genus-zero part of this CohFT determines a generically semisimple Frobenius manifold, described by the following structures.

\subsubsection{The inner product}

The inner product of the Frobenius manifold is defined by the following equation:
\[
\eta(\phi_a,\phi_b) := \left\langle\left\langle\phi_{a}\;\phi_{b}\; \phi_{0}  \right\rangle\right\rangle^\lambda_{0,3} =\left\langle\phi_{a}\;\phi_{b}\; \phi_{0}  \right\rangle^\lambda_{0,3},
\]
where the second equality follows from the fact that the genus-zero primary invariants vanish if $m_i=1$ for some $i$ and $n>3$. We use $\phi^a$ to denote the dual of $\phi_a$ under this inner product.

\subsubsection{The quantum product}

The quantum product in the small phase space, denoted $\bullet_\tau$, is defined by the equation
\[
\eta(\phi_a\bullet_\tau\phi_b,\phi_c):=\left\langle\left\langle\phi_{a}\;\phi_{b}\;\phi_{c}  \right\rangle\right\rangle^\lambda_{0,3}.
\]
As we will see in Section \ref{sec:comps}, the quantum product is semisimple for generic $\tau$.

\subsubsection{The canonical coordinates}

Since the quantum product is generically semisimple, we can find an idempotent basis $\{e_\alpha:\alpha=0,\dots,4\}$:
\[
e_\alpha\bullet_\tau e_\beta=\delta_{\alpha,\beta}e_\alpha.
\]
Let $\mathbf u=\{u^\alpha\}$ be the canonical coordinates, defined by the equation
\[
\sum_\alpha  e_\alpha d u^\alpha =\sum_i\phi_i d\tau^i=\phi_1d\tau,
\]
and normalized such that $u^\alpha(\tau=0)=0$. We define the normalized basis by
\[
\tilde e_\alpha:=\Delta_\alpha^{1/2}e_\alpha
\]
where
\[
\Delta_\alpha:=\eta(e_\alpha,e_\alpha)^{-1},
\]
which form an orthonormal basis for the quantum product. Let $\Psi$ denote the transition matrix between the bases $\phi_i$ and $\tilde e_\alpha$, which, by orthogonality, we can write as
\[
\Psi_{\alpha j}  = (\tilde e_\alpha,  \phi_j) ,\quad  j,\alpha = 0,1,2,3,4.
\]

\subsubsection{The fundamental solution}
The quantum product can be used to define the quantum differential equation on the Frobenius manifold (see Givental \cite{Giv} or Lee-Pandharipande \cite{LP} for details). The following important result of Givental describes a fundamental solution to the quantum differential equation in normalized canonical coordinates, and defines the $R$-matrix of the Frobenius manifold (up to a constant).

\begin{theorem}[Givental \cite{Giv}]\label{thm:rmatrix}
There exist fundamental solution matrices in canonical coordinates of the form
\[
R(\mathbf u,z)\re^{U/z}
\]
where $U=\mathrm{diag}(u^0,\dots,u^4)$ and $R(\mathbf u,z)=(1+R_1(\mathbf u)z+R_2(\mathbf u)z^2+\dots)$ satisfies the unitary condition
\[
R(\mathbf u,z)R(\mathbf u,-z)^*=1.
\]
Moreover, such an $R(\mathbf u,z)$ is unique up to right-multiplication by a unitary matrix of the form $\exp\left(\sum_{k\geq 0}a_{2k+1}z^{2k+1}\right)$ where $a_k=\mathrm{diag}(a_k^0,\dots,a_k^4)$ are constants.
\end{theorem}

\begin{remark}
If the Frobenius manifold has an Euler vector field, then the R-matrix can be determined uniquely by imposing that it be homogeneous. However, since there does not exist a Euler vector field in our case, we will need to normalize the R-matrix by hand.
\end{remark}

\subsection{The genus-one formula}

We use the following genus-one formula to derive Theorem \ref{thm:twisted}. This formula was proved in the conformal case by Dubrovin--Zhang \cite{DZ} and in the torus-equivariant GW setting by Givental \cite{Giv}. In general, the formula follows from the Givental--Teleman reconstruction theorem \cite{Teleman}, as we show in Appendix \ref{appendix:genusoneformula}.

\begin{theorem}[Dubrovin--Zhang \cite{DZ}, Givental \cite{Giv}, Teleman \cite{Teleman}] \label{Giv-DZ} There exists an R-matrix, which we denote by $R^\lambda$, satisfying the conditions of Theorem \ref{thm:rmatrix}, such that the genus-one potential
\[
F_1^\lambda(\tau):=\langle\langle - \rangle\rangle_{1,0}^\lambda
\]
is given by
\begin{equation}\label{eq:genusoneformula}
d F_1^\lambda = \sum_\alpha \left( \frac{1}{48}d \log \Delta_\alpha  +  \frac{1}{2} (R_1^\lambda)_{\alpha\alpha} du^\alpha \right) .
\end{equation}
\end{theorem}

\begin{remark}
Ensuring that Equation \eqref{eq:genusoneformula} holds at $\tau=0$ allows us to normalize the ambiguous constant factor of Theorem \ref{thm:rmatrix}.
\end{remark}

\begin{remark}
Theorem \ref{Giv-DZ} requires the semi-simplicity of the quantum product for generic choice of $\tau$. We verify this property in the next section.
\end{remark}

Applying Theorem \ref{Giv-DZ}, we can prove Theorem \ref{thm:twisted} by computing $\Delta_\alpha$ and $(R_1^\lambda)_{\alpha\alpha}$ explicitly. We carry out these computations in the next section.

\section{Genus-zero computations}\label{sec:comps}

In this section we provide explicit computations of the Frobenius manifold data introduced in Section \ref{sec6}.

\subsection{The inner product}

The inner product is determined by the following lemma.

\begin{lemma}
In the basis $\{\phi_i\}$,
\[
\eta = \left(\begin{matrix}
&  & & 5 & \\
& & 5 & & \\
& 5 & &  & \\
 5 & & & & \\
 & & & & 5\lambda^{5}
\end{matrix}\right).
\]
\end{lemma}
\begin{proof}
This follows immediately from the definitions.
\end{proof}

\subsection{The quantum product}

The following Proposition determines the quantum product in terms of two to-be-determined coefficients $f$ and $g$.

\begin{proposition}\label{prop:quantumproduct}
The quantum product takes on the form
\[
\phi_0 \, \bullet_\tau  =  \left(\begin{matrix}
1 &  & &   & \\
& 1&  & & \\
&   &1 &  & \\
  & & & 1& \\
 & & & & 1
\end{matrix}\right) ,\quad \phi_1 \, \bullet_\tau = \left(\begin{matrix}
&  & &  & \lambda^{5} g\\
1 & &  & & \\
& f & &  & \\
 & & 1& &  \\
 & & & g&  \\
\end{matrix}\right),
\]
where $f,g\in\bQ[[\tau]]$ are monic. Moreover, by the associativity of the quantum product, we can write
\begin{align*}
& \phi_2 \bullet_\tau  \phi_2  = \frac{g}{f} \cdot \phi_4,\quad
\phi_2 \bullet_\tau  \phi_3  = \lambda^{5} \frac{g^2}{f} \cdot \phi_0,\quad
\phi_2 \bullet_\tau  \phi_4  = \lambda^{5} \frac{g}{f} \cdot \phi_1, \\
& \phi_3 \bullet_\tau  \phi_3  = \lambda^{5} \frac{g^2}{f} \cdot \phi_1,\quad
\phi_3 \bullet_\tau  \phi_4  = \lambda^{5} g \cdot \phi_2,\quad
\phi_4 \bullet_\tau  \phi_4  = \lambda^{5} \cdot \phi_3.
\end{align*}
\end{proposition}
\begin{proof}
The quantum product is defined by the structure constants
\[
\eta(\phi_a\bullet_\tau\phi_b,\phi_c)=\left\langle\left\langle\phi_{a}\;\phi_{b}\;\phi_{c}  \right\rangle\right\rangle^\lambda_{0,3} = \sum_{k\geq 0} \frac{\tau^k}{k!}\left\langle\phi_{a}\;\phi_{b}\;\phi_{c}\;\phi_{1} \cdots\phi_{1}  \right\rangle^\lambda_{0,3+k}.
\]
By \eqref{eq:equichern}, the correlators vanish unless
\[
k= \dim \M_{0,(a+1,b+1,c+1,2^k)}^{1/5}  =d-5s = a+b+c +k -3-5s
\]
for some non-negative integer $s$. This happens when the characteristic class contributes $\lambda^{5s} c_{d-5s}$ to the integration. 

If $a=0$, then the only possibilities are:
\begin{align*}
s=0&\text{ and } b+c=3 ;\\
s=1&\text{ and } b=c=4 .
\end{align*}
The corresponding correlators are determined by the pairing, and they give the desired form of the matrix $\phi_0\bullet_\tau$.

If $a=1$, then the only possibilities are:
\begin{align*}
s=0& \text{ and } \{b,c\}=\{0,2\} ;\\
s=0& \text{ and } b=c=1 ;\\
s=1&\text{ and } \{b,c\}=\{3,4\} .
\end{align*}
The first case is determined by the pairing, and gives the desired form for $\phi_1\bullet_\tau\phi_0$ and $\phi_1\bullet_\tau\phi_2$. The rest of the matrix $\phi_1\bullet_\tau$ is then determined by the following correlators:
\[
f:=\frac{1}{5}\langle\langle\phi_1\;\phi_1\;\phi_1\rangle\rangle^\lambda_{0,3}
\]
and
\[
g:=\frac{[\lambda^5]}{5}\langle\langle\phi_1\;\phi_3\;\phi_4\rangle\rangle^\lambda_{0,3}.
\]
\end{proof}

In order to explicitly compute $f$ and $g$, we still require some more work. Start by defining the small I-function for the twisted invariants:
\begin{align*}
I^\lambda(t,z)&:=z\sum_{a\geq 0}\frac{t^{a}}{z^aa!}\prod_{0< k<\frac{a+1}{5} \atop \langle k \rangle = \langle \frac{a+1}{5}\rangle}(\lambda^5+(kz)^5)\phi_{a}\\
&=:\sum_{k\geq 0} I_k(t) z^{-k+1}\phi_k.
\end{align*}
Notice that $I^\lambda(t,z)_{\phi_4=\lambda=0}=I(t,z)$, where the latter is the FJRW I-function defined in the introduction. Moreover, the definitions of $I_0(t)$ and $I_1(t)$ here agree with those given in the introduction. Also, $t I^\lambda(t,z)$ is annihilated by the following Picard-Fuchs operator
\begin{equation}\label{eq:PF}
 \left( \frac{1}{5}t\frac{d}{dt}\right)^5 + \left(\frac{\lambda}{z}\right)^5 - t^{-5}\prod_{k=1}^5\left(t\frac{d}{dt} -k\right) .
\end{equation}
Setting $\tau = \frac{I_1(t)}{I_0(t)}$, as in the introduction, the genus-zero mirror theorem for the twisted invariants, proved by Chiodo--Ruan \cite{CR}, states that
\begin{equation}\label{eq:smallmirror}
J^\lambda(\tau,z)  = \frac{I^\lambda(t,z)}{I_0(t)} .
\end{equation}
where
\begin{align*}
J^\lambda(\tau,z)&:=z\phi_0+ \tau \phi_1 + \sum_{a=0}^4\left\langle\left\langle\frac{\phi_a}{z-\psi}\right\rangle\right\rangle^\lambda_{0,1}\phi^a\\
&=:\sum_{k\geq 0} J_k(\tau) z^{-k+1}\phi_k
\end{align*}

Following Zagier--Zinger \cite{ZZ}, for any $F(t,z) \in \bQ[[t,z^{-1}]]$, we introduce the following Birkhoff factorization operator
\begin{equation}\label{eq:bop}
\mathbf M F(t,z) :=  z  \frac{d}{dt}\frac{F(t,z)}{F(t,\infty)} .
\end{equation}
If $F(t,z)$ is a state-space-valued series, then we set $\phi_i=1$ in the denominator of \eqref{eq:bop}. We inductively define series $\I_{p,q}(t)$ by
\begin{align}\label{eq:icoeffs}
\I_{0,q}(t) := I_q(t),\quad \text{ and }\quad \I_{p,q}(t) := \frac{d}{dt} \left( \frac{\I_{p-1,q}(t)}{\I_{p-1,p-1}(t)} \right) \, \text{ for $q\geq p>0$, }
\end{align}
so that
\begin{equation}\label{eq:Moperator}
\mathbf  M^p  (\I^\lambda(t,z)/z) = \sum_{q\geq 0} \I_{p,p+q}(t) z^{-q} \phi_{p+q}
\end{equation}
for $p\geq 0$.

\begin{example}\label{example1}
We have
\[
 \I_{1,1} = \frac{d}{dt} \frac{I_1}{I_0} =  \frac{d\tau }{dt} ,\quad
\I_{1,2} =  \frac{d}{dt} \frac{I_2}{I_0}=  \frac{d}{dt} {J_2(\tau)} ,\quad
\I_{2,2} =  \frac{d}{dt}\frac{d}{d\tau} {J_2(\tau)}
\]
\end{example}

The Birkhoff factorization operator provides us with a convenient way to write general points of Givental's Lagrangian cone in terms of the I-function. For example, the following result allows us to compute the twisted S-operator\footnote{The superscript $*$ is used to denote that this operator is adjoint to a fundamental solution matrix $S^\lambda(\tau,z)$ of the quantum differential equation \cite{Giv}.}, defined by
\[
S^\lambda(\tau,z)^*(\phi_k):=\phi_k+\sum_{m=0}^4\left\langle\left\langle\phi_k\;\frac{\phi_m}{z-\psi}\right\rangle\right\rangle_{0,2}^\lambda\phi^m
\]
in terms of the I-function.

\begin{lemma}\label{lem:birkhoff}
The twisted S-operator in the small phase space is given by
\[
S^{\lambda}(\tau,z)^*(\phi_k) = \frac{ {\mathbf M}^k (\I^\lambda(t,z)/z) }{\I_{k,k}}.
\]
where $\tau=\tau(t)$ is the mirror map.
\end{lemma}

\begin{proof}
Due to the basic properties of the Lagrangian cone $\mathcal L$, proved by Givental \cite{GiventalSymplectic}, it follows that $I^\lambda(t,-z)$ and $J^\lambda(\tau,-z)$ lie in the same tangent space $T$ of $\mathcal L$. Moreover, a basis for the intersection $\mathcal L\cap T$, over the ring of polynomials in $z$, is provided by
\begin{equation}\label{eq:conebasis}
\left\{zS^{\lambda}(\tau,-z)^*(\phi_k)\right\}.
\end{equation}
It also follows from the same basic properties of $\mathcal L$ that
\[
\frac{z {\mathbf M}^k (\I^\lambda(t,z)/z) }{\I_{k,k}}\in \mathcal L \cap T.
\]
Therefore, it can be written as a linear combination of the basis \eqref{eq:conebasis} with coefficients taken from the ring of polynomials in $z$. The lemma follows by observing that
\[
\frac{ z{\mathbf M}^k (\I^\lambda(t,z)/z) }{\I_{k,k}}=z\phi_k+z\cO(z^{-1})
\]
and
\[
zS^{\lambda}(\tau,z)^*(\phi_k)=z\phi_k+z\cO(z^{-1}).
\]
\end{proof}

We now use Lemma \ref{lem:birkhoff} to give explicit formulas for $f$ and $g$, and thus finish the computation of the quantum product.

\begin{proposition} \label{prop:key}
For $j=0,1,2,3,4$, we have
\[
 \langle\langle \phi_1\;\phi_i\;\phi_j\rangle\rangle_{0,3}^\lambda =\eta(\phi_i,\phi_{j+1})\frac{\I_{j+1,j+1}}{\I_{1,1}}.
\]
In particular,
\begin{equation}
f = \frac{ \I_{2,2}}{ \I_{1,1}} ,\quad
g = \frac{ \I_{4,4}}{ \I_{1,1}}  .
\end{equation}
\end{proposition}

\begin{proof}
Notice that $\langle\langle \phi_1\;\phi_i\;\phi_j\rangle\rangle_{0,3}^\lambda$ can be obtained by differentiating the two point correlator $\langle\langle \phi_i\;\phi_j\rangle\rangle_{0,2}^\lambda$. By definition of the $S$-operator, the $z^{-1}$-coefficient records these two point correlators:
\[
\langle\langle \phi_i \; \phi_j \rangle\rangle^\lambda_{0,2}=  \eta(\phi_i,S^\lambda(\tau,z)^*(\phi_j))[z^{-1}]=\eta\left(\phi_i,\phi_{j+1}\right)\frac{\I_{j,j+1}}{\I_{j,j}},
\]
where we have used Lemma \ref{lem:birkhoff} and Equation \eqref{eq:Moperator} in the final equality. Differentiating both sides, we obtain
\begin{align*}
\langle\langle \phi_1\;\phi_i\;\phi_j\rangle\rangle_{0,3}^\lambda&=\eta\left(\phi_i,\phi_{j+1}\right)\frac{d}{d\tau}\left(\frac{\I_{j,j+1}}{\I_{j,j}}\right)\\
&=\frac{\eta\left(\phi_i,\phi_{j+1}\right)}{\I_{1,1}}\frac{d}{dt}\left(\frac{\I_{j,j+1}}{\I_{j,j}}\right)\\
&=\eta\left(\phi_i,\phi_{j+1}\right)\frac{\I_{j+1,j+1}}{\I_{1,1}}
\end{align*}
where the second equality used Example \ref{example1} and the final equality uses the recursive definition \eqref{eq:icoeffs}.
\end{proof}

The series $\I_{p,p}$ satisfy the following important properties.

\begin{lemma} \label{lem:Ipp}
Define  $L := \left(1-\frac{t^5}{5^5}\right)^{-\frac{1}{5}}$. The following identities hold:
\begin{enumerate}[(i)]
\item $\I_{0,0} \I_{1,1}  \cdots \I_{4,4}  = L^5$; \label{one}
\item $I_{5+p,5+p} =  \lambda^5\I_{p,p}$; \label{two}
\item for $0\leq p\leq 4$, $I_{p,p}=I_{4-p,4-p}.$ \label{three}
\end{enumerate}
\end{lemma}
\begin{proof}
The proof closely follows techniques developed in Zagier--Zinger \cite{ZZ}.

Define
\[
\widetilde I^\lambda(t,z):=\phi_4\lambda^{-5}+\sum_{a\geq 0}\frac{t^{a+1}}{z^{a+1}(a+1)!}\prod_{0< k<\frac{a+1}{5} \atop \langle k \rangle = \langle \frac{a+1}{5}\rangle}(\lambda^5+(kz)^5)\phi_a
\]
so that
\[
\mathbf M\widetilde I^\lambda(t,z)=\lambda^5I^\lambda(t,z)/z.
\]
Define
\begin{align*}
G(t,z)&:=\widetilde I^\lambda(t,z);\\
F(t)&:=G(t,\infty);\\
A(t,z)&:=-\left(\lambda/z\right)^5G(t,z);\\
D&:=t\frac{d}{dt}.
\end{align*}
It is straightforward to check the following analog of the Picard-Fuchs equation \eqref{eq:PF}:
\begin{equation}\label{eq:PF2}
\left[(D/5)^5-t^{-5}\prod_{k=0}^4\left(D-k\right) \right]G=A.
\end{equation}
It is also apparent from the definitions of $F$ and $A$, along with \eqref{eq:PF2}, that
\begin{equation}\label{eq:PF3}
\left[(D/5)^5-t^{-5}\prod_{k=0}^4\left(D-k\right) \right]F=0.
\end{equation}
Expanding the product, write
\[
\left[(D/5)^5-t^{-5}\prod_{k=0}^4\left(D-k\right) \right]=\sum_{r=0}^5C_r(t)D^r.
\]
By writing $G=(G/F)F$ and applying the product rule to \eqref{eq:PF2}, it is straightforward to show that
\begin{equation}\label{eq:PF4}
\sum_{s=0}^4 C^{(1)}_s(t)D^s G^{(1)}=zA,
\end{equation}
where
\[
C^{(1)}_s(t):=\sum_{r=s+1}^5{r\choose s+1}C_r(t)D^{r-1-s}F
\]
and
\[
G^{(1)}(t,z):=zD(G/F)=t\cdot\mathbf M G(t,z).
\]
The initial term in \eqref{eq:PF4} vanishes by \eqref{eq:PF3}. By iterating this process $p$ times, for $1\leq p\leq 5$, we obtain
\begin{equation}\label{eq:PF5}
\sum_{s=0}^{5-p} C^{(p)}_s(t)D^s G^{(p)}=z^pA
\end{equation}
where the coefficients are defined recursively:
\[
C^{(p)}_s(t):=\sum_{r=s+1}^{6-p}{r\choose s+1}C_r^{(p-1)}(t)D^{r-1-s}G^{(p-1)}(t,\infty)
\]
and
\[
G^{(p)}(t,z):=zD\left(G^{(p-1)}(t,z)/G^{(p-1)}(t,\infty)\right)=t\cdot\mathbf M^pG(t,z).
\]
The top coefficients are easily computed:
\begin{align*}
C^{(1)}_4&=C_5\lambda^{-5};\\
C^{(2)}_3&=C_5tI_{0,0};\\
&\vdots\\
C^{(5)}_0&=C_5t^4I_{0,0}\cdots I_{3,3}.
\end{align*}
Moreover, $C_5=1/5^5-t^{-5}$. Thus, for $p=5$, \eqref{eq:PF5} becomes
\begin{equation}\label{PF6}
\left(\left(t/5\right)^5-1\right)I_{0,0}\cdots I_{3,3}\mathbf M^5 (\widetilde I^\lambda(t,z))=-\lambda^5 \widetilde I^\lambda(t,z).
\end{equation}
Setting $z=\infty$, we obtain
\[
\left(\left(t/5\right)^5-1\right)I_{0,0}\cdots I_{4,4}=-1,
\]
which proves \eqref{one}. To prove \eqref{two}, re-insert \eqref{one} into \eqref{PF6} to obtain
\[
\frac{\mathbf M^4 (I^\lambda(t,z)/z)}{I_{4,4}}=\widetilde I^\lambda(t,z).
\]
Applying $z\frac{d}{dt}$ to both sides proves that
\[
\mathbf M^5 (I^\lambda(t,z)/z)=\lambda^5I^\lambda(t,z)/z.
\]
To prove \eqref{three}, we use Proposition \ref{prop:key} and the symmetry of the three-point functions:
\[
5=\langle\langle\phi_1\;\phi_2\;\phi_0 \rangle\rangle_{0,3}^\lambda=\langle\langle\phi_1\;\phi_0\;\phi_2 \rangle\rangle_{0,3}^\lambda=5\frac{I_{3,3}}{I_{1,1}}
\]
and
\[
5\lambda^{5}\frac{I_{4,4}}{I_{1,1}}=\langle\langle\phi_1\;\phi_4\;\phi_3 \rangle\rangle_{0,3}^\lambda=\langle\langle\phi_1\;\phi_3\;\phi_4 \rangle\rangle_{0,3}^\lambda=\frac{I_{5,5}}{I_{1,1}}=5\lambda^5\frac{I_{0,0}}{I_{1,1}},
\]
where the final equality uses \eqref{two}.
\end{proof}

\begin{example}\label{example2}[Yukawa coupling]  {Set
\[
Y:=\frac{1}{5}\langle\langle \phi_1\; \phi_1\; \phi_1 \rangle\rangle^\w_{0,3}.
\]
By Lemma \ref{lem:Ipp},
\[
Y=\frac{\I_{2,2}}{\I_{1,1}} =
\frac{\I_{0,0}^2 \I_{1,1}^2 \I_{2,2}}{ I_0^2}\cdot  \I_{1,1}^{-3}= \frac{L^5}{I_0^2} \left(\frac{dt}{d\tau}\right)^3 ,
\]
which coincides with the well-known result (see, for example, equation (3.65) in \cite{HKQ})}.
\end{example}

\subsection{Canonical coordinates} Consider the following normalizations:
\begin{align*}
\tilde \phi_0= \phi_0,\quad \tilde \phi_1 = \frac{g^{-\frac{2}{5}} f^{-\frac{1}{5}}}{\lambda} \cdot \phi_1,\quad &
\tilde \phi_2 = \frac{g^{-\frac{4}{5}} f^{\frac{3}{5}}}{\lambda^{2}} \cdot \phi_2,\\
\tilde \phi_3 = \frac{g^{-\frac{6}{5}} f^{\frac{2}{5}}}{\lambda^{3}} \cdot  \phi_3,\quad &
\tilde \phi_4 = \frac{g^{-\frac{3}{5}} f^{\frac{1}{5}}}{\lambda^{4}} \cdot  \phi_4.
\end{align*}
As usual, we compute all indices modulo five. By Proposition \ref{prop:quantumproduct}, we have
\[
\tilde \phi_i \bullet_\tau \tilde \phi_j = \tilde \phi_{i+j} .
\]
Hence, the quantum product is semisimple and we can define a canonical basis by
\[
e_\alpha = \frac{1}{5}\sum_i  \tilde \phi_i \xi^{-i \alpha} ,\quad  \alpha = 0,1,2,3,4 ,
\]
so that
\[
e_\alpha \bullet e_\beta = \delta_{\alpha \beta} e_\alpha .
\]
Let $\{u^\alpha\}$ be the canonical coordinates:
\[
\sum_\alpha  e_\alpha d u^\alpha =\phi_1d\tau,
\]
normalized by the convention that $u^\alpha(\tau=0)=0$.

\begin{lemma}\label{lem:du}
We have
\[
d u^\alpha = \xi^\alpha\lambda \cdot du,
\]
where
\[
du =  g^{2/5}f^{1/5}d\tau = L dt.
\]
\end{lemma}
\begin{proof}
This follows from the definition of $\tilde \phi_i$ in terms of $e_\alpha$, along with Proposition \ref{prop:key} and Lemma \ref{lem:Ipp}.
\end{proof}

The normalized canonical coordinates are defined by
\[
\tilde e_\alpha:=\Delta_\alpha^{1/2} e_\alpha
\]
where
\[
\Delta_\alpha=\eta(e_\alpha,e_\alpha)^{-1}.
\]
\begin{lemma}\label{lem:delta}
We have
\[
\Delta_\alpha = (\xi^\alpha \lambda)^3 \frac{g^{6/5}}{f^{2/5}} = (\xi^\alpha \lambda)^3 \frac{ I_0^2}{L^2}
\]
\end{lemma}
\begin{proof}
This follows from the definition of $e_\alpha$ in terms of $\phi_i$, the computation of $\eta$ in the $\phi_i$-basis, Proposition \ref{prop:key}, and Lemma \ref{lem:Ipp}.
\end{proof}

The transition matrix between flat coordinates and normalized canonical coordinates is defined by
\[
{\Psi_\alpha}_ j  = \eta(\tilde e_\alpha,  \phi_j) ,\quad  j,\alpha = 0,1,2,3,4.
\]
From the definition of $e_\alpha$ in terms of $\phi_i$, we have
\[
{\Psi_\alpha}_j =\xi^{\alpha (j-3/2)} c_{3-j},
\]
where $c_i$ are given by
\begin{align*}
& c_{-1} = \lambda^{\frac{5}{2}},  \quad c_0 = \lambda^{\frac{3}{2}}\frac{I_{0,0}}{L} , \quad c_1= \lambda^{\frac{1}{2}}\frac{I_{0,0}I_{1,1}}{L^2}\\
& c_2 = c_1^{-1}=\lambda^{-\frac{1}{2}}\frac{I_{0,0}I_{1,1}I_{2,2}}{L^3}, \quad c_3 = c_0^{-1}=\lambda^{-\frac{3}{2}}\frac{I_{0,0}I_{1,1}I_{2,2}I_{3,3}}{L^4}
\end{align*}
For convenience, we also define $c_{4}:=\lambda^{-\frac{5}{2}}$, so that $c_j=c_{3-j}^{-1}$ for $j=0,\dots,4$. The inverse matrix of $\Psi$ is given by
\[
\Psi^{-1}_{j \alpha}   = \frac{\xi^{\alpha(3/2-j)}}{5} c_j .
\]
We define
\[
dC^{\gamma} := \sum_{j=0}^4 \xi^{\gamma (j-3/2)}c_{j}^{-1} dc_j =  \sum_{j=0}^4 \xi^{\gamma (j-3/2)}d\log c_j
\]
so that
\begin{equation} \label{eq:R1offdiaga}
{(\Psi d\Psi^{-1})_{\alpha}}_\beta =   \frac{1}{5}\sum_{j=0}^4 \xi^{(\alpha-\beta) (j-3/2)}c_{3-j} dc_j
= \frac{1}{5}dC^{\alpha-\beta}.
\end{equation}
Notice that $c_4$ does not contribute to $dC^\gamma$, and $dC^{\gamma+5}=dC^\gamma$.

\subsection{The fundamental solution}

In this subsection, we prove the following result.

\begin{proposition} \label{prop:R1}The diagonal entries of the $R_1^\lambda$-matrix are given by
\[
{(R_1^\lambda)_{\alpha\alpha}}  =  \frac{1}{5\lambda\xi^\alpha} \frac{d}{du} \left(\frac{5}{4}\log(L)-4 \log(I_0)-\log( \I_{1,1}) \right) +C_\alpha
\]
where $\sum_{\alpha}\xi^\alpha C_\alpha=0$.
\end{proposition}

\begin{proof}

Following Givental \cite{Giv} (see also Lee-Pandharipande's book \cite{LP} for more details), we can compute the $R_1^\lambda$-matrix in terms of the Frobenius manifold data  in two steps. The first step is to compute the off-diagonal components. By flatness of the quantum connection, Givental argues that
\[
\Psi d \Psi^{-1} = [dU, R_1^\lambda].
\]
For $\alpha \neq \beta$, this equation reads
\[
{(\Psi d\Psi^{-1})_{\alpha}}_\beta = (du^\alpha-du^\beta) (R_1^\lambda)_{\alpha\beta}
\]
By Lemma \ref{lem:du} and Equation \eqref{eq:R1offdiaga}, we obtain
\begin{equation}\label{eq:R1offdiag}
{(R_1^\lambda)_{\alpha\beta}} = \frac{dC^{\alpha-\beta}}{5(\xi^{\alpha}-\xi^{\beta}) \lambda du },\quad \text{for $\alpha \neq \beta$ }.
\end{equation}

The second step is to solve for the diagonal components of $R_1^\lambda$. Givental argues that the diagonal components of $dR_1^\lambda$ satisfy the following equation:
\[
{(dR_1^\lambda)_{\alpha\alpha}} = \sum_\beta (d u^{\beta}-d u^{\alpha}){(R_1)_{\alpha\beta}} {(R_1)_{\beta\alpha}}.
\]
By Lemma \ref{lem:du} and Equation \eqref{eq:R1offdiag}, we compute
\begin{align*}
{(dR_1^\lambda)_{\alpha\alpha}}&=\sum_{\beta\neq \alpha}\frac{1}{25(\xi^\alpha-\xi^\beta)} \frac{dC^{\alpha-\beta}dC^{\beta-\alpha}}{\lambda du}\\
& = \sum_{\gamma\neq 0}\frac{1}{25 \xi^\alpha(1-\xi^\gamma)} \frac{dC^{-\gamma}dC^{\gamma}}{\lambda du}  \\
&= \sum_{i,j=0,1,\cdots,4 \atop \gamma\neq 0}\frac{\xi^{(i-j)\gamma}}{25 \lambda \xi^\alpha(1-\xi^\gamma)} \frac{d \log c_i}{du}  \frac{d \log c_j}{du} du \\
& =  \frac{1}{25\lambda \xi^\alpha}\sum_{i,j}\left(2-5\left\langle\frac{j-i}{5}\right\rangle\right)\frac{d \log c_i}{du}  \frac{d \log c_j}{du} du,
\end{align*}
where, in the last step, we have used the identity
\[
\sum_{\gamma=1,2,3,4}\frac{\xi^{k\gamma}}{1-\xi^\gamma} = 2-5 \left\langle-\frac{k}{5}\right\rangle,\quad  \forall k\in \mathbb Z.
\]
Using the fact that
\[
\sum_{i=0}^4 d\log c_i = \sum_{i=0}^3 d\log c_i = d\log(c_0\cdots c_3)=0,
\]
we compute
\begin{align*}
{(dR_1^\lambda)_{\alpha\alpha}}
&=\frac{1}{25\lambda \xi^\alpha}\sum_{i,j}\left(-5\left\langle\frac{j-i}{5}\right\rangle\right)\frac{d \log c_i}{du}  \frac{d \log c_j}{du} du\\
&=-\frac{1}{25\lambda \xi^\alpha}\sum_{i>j} \frac{d \log c_i}{du}  \frac{d \log c_j}{du} du.
\end{align*}
Recalling that $d\log(c_j)=-d\log(c_{3-j})$, we have
\begin{align}\label{eq:dR1}
{(dR_1^\lambda)_{\alpha\alpha}}\nonumber
&=-\frac{1}{5\lambda \xi^\alpha}\sum_{(i,j)=(2,1),(3,0)}\frac{d \log c_i}{du}  \frac{d \log c_j}{du} du \nonumber
\\&=\frac{1}{5\lambda \xi^\alpha} \left(\left( \frac{d \log  c_2}{du}\right)^{2} + \left( \frac{d \log c_3}{du}\right)^{2} \right) du.
\end{align}
The next lemma manipulates Formula \eqref{eq:dR1} into the form of Proposition \ref{prop:R1}.

\begin{lemma}\label{lem:identities}
We have
\begin{equation} \label{eq:main}
 \left( \frac{d \log  c_2}{du}\right)^{2} + \left( \frac{d \log c_3}{du}\right)^{2}
 =  \frac{d^2}{du^2} \left(\frac{5}{4}\log(L)-4 \log(I_0)-\log( \I_{1,1}) \right)
\end{equation}
\end{lemma}
\begin{proof}
For any function $F$ of $t$, we use the notation $F' := \frac{d}{dt}F$. Recall that
\[
c_2 =\lambda^{-\frac{1}{2}} \frac{L^2}{I_0\I_{1,1}} ,
\quad
c_3  = \lambda^{-\frac{3}{2}} \frac{L}{I_0},
\]
Using the facts that $d u = L dt$ and  $t L' = L(L^5-1)$, we have
\[
L^{2}  \frac{d^2}{du^2} =  -\frac{(L^5-1)}{t}\frac{d}{dt} + \frac{d^2}{dt^2}
,\quad
 \frac{5}{4} \frac{d^2}{du^2} \log(L) = \frac{5(L^{5}-1)L^5}{(tL)^2} .
\]
From this, we compute that Equation \eqref{eq:main} is equivalent to
\[
5\,{\frac {{L^5}-1}{{t}^{2}}}+{\frac {{L^5}-1}{t} \left( 10\,{\frac {{I_0}'
}{{I_0}}} +5\,{\frac {{{\I_{1,1}}}'
}{{{\I_{1,1}}}}}  \right) }-2\,{\frac {{{\I_{1,1}}}'
}{{{\I_{1,1}}}}}  {\frac {{I_0}'
}{{I_0}}} +2\,{\frac {{{I_0}'}^2
}{{I_0}^2}} -4\,{\frac {{I_0}''
}{{I_0}}} -{\frac {{{\I_{1,1}}}''
}{{{\I_{1,1}}}}} =0 .
\]
We can rewrite this equality in the following form
\begin{equation}\label{eq:main1}
5\,{\frac {{L^5}-1}{{t}^{2}}}+{\frac {{L^5}-1}{t} }\frac{5(I_0 ^2 {\I_{1,1}})'}{I_0^2 {\I_{1,1}}} =  \frac{(I_0^2 {\I_{1,1}})''- 2(I_0^2{\I_{1,1}})'\frac{I_0'}{I_0}+2I_0''I_0{\I_{1,1}}}{I_0^2 {\I_{1,1}}} .
\end{equation}
Since $\I_{1,1} = \frac{d}{dt}\frac{I_1}{I_0}$, we can rewrite both sides of Equation \eqref{eq:main1} as
\[
RHS = \frac{I_1'''I_0 -I_1I_0''' -I_1''I_0' + I_1'I_0''  }{I_1'I_0 -I_1I_0'}
\]
and
\[
LHS =
5\,{\frac {L^5-1}{{t}^{}}} \left( \frac{1}{t}+\frac{I_1''I_0 -I_1I_0''}{I_1'I_0 -I_1I_0'}
\right) .
\]
Notice that $  {L^5-1}  = \frac{{t^5}}{{5^5}-{t^5}} $. Thus, Equation \eqref{eq:main1} is equivalent to
\begin{equation}\label{eq:main2}
(5^5-t^5) \cdot  \clubsuit =  5t^5\cdot  \spadesuit
 \end{equation}
where
 \begin{align*}
 \clubsuit :=&t\cdot (I_1'''I_0 -I_1I_0''' -I_1''I_0' + I_1'I_0'' ),\\
 \spadesuit :=& I_1''I_0 -I_1I_0'' + t^{-1}(I_1'I_0 -I_1I_0').
 \end{align*}

In order to prove Equation \eqref{eq:main2}, first recall that, for $k=0,\dots,4$,
\[
I_{k} =  \sum_{d\geq 0}\frac{\left(\frac{k+1}{5}\left(\frac{k+1}{5}+1\right)\cdots \left(\frac{k-4}{5}+d\right)\right)^5}{ (k+5d)!} t^{5d+k} .
\]
Define
\[
{\mathbf C}_{d_1,d_2}:=\frac{\left(\frac{1}{5}\cdot \frac{6}{5} \cdots \left(d_1-\frac{4}{5}\right)\right)^5}{ (5d_1)!}
\frac{\left(\frac{2}{5}\cdot \frac{7}{5} \cdots \left( d_2-\frac{3}{5}\right)\right)^5}{ (5d_2+1)!} t^{5d_1+5d_2-1}.
\]
By definition, we have
\begin{align*}
\clubsuit = &\sum_{d_1,d_2\geq 0} {\mathbf C}_{d_1,d_2}
 \cdot \big( (5d_2+1)(5d_2)(5d_2-1)-(5d_1)(5d_1-1)(5d_1-2) \\
&\qquad \quad   -(5d_2+1)(5d_2) (5d_1)+(5d_2+1)(5d_1)(5d_1-1)  \big)\\
=&-\sum_{d_1,d_2\geq 0}{\mathbf C}_{d_1,d_2}
  \cdot5 ( 5 d_1-5 d_2-1) (5d_1^2+5d_2^2-3d_1-d_2)
\end{align*}
and, similarly,
\begin{align*}
\spadesuit =&-\sum_{d_1,d_2\geq 0}{\mathbf C}_{d_1,d_2}
  \cdot  ( 5 d_1-5 d_2-1) (5d_1 +5d_2 +1).
\end{align*}
Therefore,
\[
5\spadesuit +   \clubsuit  =-\sum_{d_1,d_2\geq 0}{\mathbf C}_{d_1,d_2}
 \cdot 5 ( 5 d_1-5 d_2-1) (5d_1^2+5d_2^2+2d_1+4d_2+1) .
\]
For $d_1,d_2\geq 0$, consider two sets of scalars $A_{d_1,d_2}$ and $B_{d_1,d_2}$ defined by
\[
A_{d_1,d_2} :=\frac{1}{5} \frac{(  d_1-\frac{4}{5})^4}{5 d_1+5 d_2-2},\quad
B_{d_1,d_2} := -\frac{1}{5} \frac{(  d_2-\frac{3}{5})^4}{5 d_1+5 d_2-2}.
\]
Notice that these scalars satisfy the following two relations:
\begin{align*}
&\frac{(5  d_1) (5 d_1-1) (5 d_1-2) (5 d_1-3)}{(d_1-\frac{4}{5})^4}\! A_{d_1,d_2}+ \frac{ (5 d_2+1)(5  d_2 )(5 d_2-1) (5 d_2-2)}{{(d_2-\frac{3}{5})^4}}\! B_{d_1,d_2}
\\&\hspace{5cm}=   ( 5 d_1-5 d_2-1) (5d_1^2+5d_2^2-3d_1-d_2)
\end{align*}
and
\[
A_{d_1+1,d_2} + B_{d_1,d_2+1} =
5^{-5}\cdot 5 ( 5 d_1-5 d_2-1) (5d_1^2+5d_2^2+2d_1+4d_2+1) .
\]

We obtain
\begin{align*}
5\spadesuit +   \clubsuit =&-\sum_{d_1,d_2\geq 0}5^5 {\mathbf C}_{d_1,d_2}
 \cdot  (A_{d_1+1,d_2} + B_{d_1,d_2+1})\\
 =&- 5^5  \left(\sum_{d_1>0,d_2\geq 0} {\mathbf C}_{d_1-1,d_2}
 A_{d_1,d_2} + \sum_{d_1\geq 0,d_2> 0}{\mathbf C}_{d_1,d_2-1} B_{d_1,d_2}\right)\\
 = & -5^5 \sum_{d_1 ,d_2\geq 0}  \Big(\frac{{\mathbf C}_{d_1,d_2}}{t^5}\frac{(5  d_1) (5 d_1-1) (5 d_1-2) (5 d_1-3) (5 d_1-4)}{(d_1-\frac{4}{5})^5} A_{d_1,d_2}\\
 & \qquad\quad + \frac{{\mathbf C}_{d_1,d_2}}{t^5} \frac{(5 d_2+1)(5  d_2) (5 d_2-1) (5 d_2-2) (5 d_2-3) }{{(d_2-\frac{3}{5})^5}} B_{d_1,d_2}\Big)
 \\
 = &-\sum_{d_1,d_2\geq 0} 5^5 t^{-5}{\mathbf C}_{d_1,d_2} 5( 5 d_1-5 d_2-1) (5d_1^2+5d_2^2-3d_1-d_2)\\
 =& 5^5 t^{-5} \cdot  \clubsuit  ,
\end{align*}
proving \eqref{eq:main2} and finishing the proof of the lemma.
\end{proof}

Equation \eqref{eq:dR1} and Lemma \ref{lem:identities} complete the computation of $(dR_1^\lambda)_{\alpha\alpha}$. In order to finish the proof of the proposition, we still need to show that $\sum_\alpha \xi^\alpha C_\alpha=0$. Consider the genus-one formula of Theorem \ref{Giv-DZ}:
\[
d F_1^\lambda = \sum_\alpha \left( \frac{1}{48}d \log \Delta_\alpha  +  \frac{1}{2} (R_1^\lambda)_{\alpha\alpha} du^\alpha \right) .
\]

Notice that $F_1^\lambda=\cO(\tau^5)$. Since $\tau=\cO(t)$, the left hand side of the genus-one formula vanishes at $t=0$. From Lemma \ref{lem:delta}, we compute that
\[
d\log\Delta_\alpha=\frac{2L^3}{I_0^5}(L'I_0-I_0'L)dt
\]
vanishes at $t=0$, since both $L'$ and $I_0'$ vanish at $t=0$. Thus, $\sum_{\alpha}(R_1^\lambda)_{\alpha\alpha}du^\alpha$ must vanish at $t=0$. From the definition of $L$, we see that $du^\alpha=5^{1/5}\xi^\alpha\lambda Ldt$ does not vanish at $t=0$, so the vanishing of $\sum_{\alpha}(R_1^\lambda)_{\alpha\alpha}du^\alpha$ is equivalent to $\sum_\alpha \xi^\alpha C_\alpha=0$. This completes the proof of the proposition.

\end{proof}

\subsection{Conclusion of Theorem \ref{thm:twisted}}

To conclude the proof of Theorem \ref{thm:twisted}, we insert the results of Lemmas \ref{lem:du}, \ref{lem:delta} and Proposition \ref{prop:R1} into the genus-one formula of Theorem \ref{Giv-DZ}. We have
\begin{align*}
d F_1^\lambda &= \sum_\alpha \left( \frac{1}{48}d \log \Delta_\alpha  +  \frac{1}{2} (R_1^\lambda)_{\alpha\alpha} du^\alpha \right)\\
&=\sum_\alpha\left(\frac{1}{24}d\log\left(\frac{I_0}{L} \right) +  \frac{1}{10} \frac{d}{du} \left(\frac{5}{4}\log(L)-4 \log(I_0)-\log( \I_{1,1}) \right)du\right)\\
&=d\log \left( I_0(t)^{\frac{5}{24}-2}\left(1-({t}/{5})^{5}  \right)^{-1/12} \left(\frac{d}{dt}\frac{ I_1(t)}{ I_0(t)}\right)^{-1/2} \right),
\end{align*}
where the last equality uses the definition of $L$ and the fact that $\I_{1,1} = \frac{d}{dt}\frac{I_1}{I_0}$. Theorem \ref{thm:twisted} now follows from the observations that $F_1^\lambda(\tau)$ has vanishing constant term.

\subsection{One-point invariants}

Here, we verify Lemma \ref{lem:onepoint}, which is the missing link between Theorem \ref{thm:twisted} and Theorem \ref{thm:mirrortheorem}. Specifically, we need to compute
\begin{equation}\label{eq:onepoint}
\langle \phi_0\psi \rangle_{1,1}^\star=\int_{\M_{1,1}}\Omega_{1,1}^\star(\phi_0)\psi_1
\end{equation}
where $\Omega_{g,n}^\star$ denotes the FJRW CohFT for $\star=\w$ and the twisted CohFT for $\star=\lambda$. For simple dimension reasons, the only part of $\Omega_{1,1}^\star(\phi_0)$ that contributes to \eqref{eq:onepoint} is the part supported in degree zero, i.e. the corresponding topological field theory element $\omega_{1,1}^\star(\phi_0)$. By the axioms of CohFTs, we have
\[
\omega_{1,1}^\star(\phi_0)=\sum_{\alpha,\beta}\eta^{\alpha,\beta}\omega_{0,3}^\star(\varphi_\alpha,\varphi_\beta,\phi_0)=\sum_{\alpha}\eta(\varphi^\alpha,\varphi_\alpha)
\]
where the sum is over all $\alpha,\beta$ in the state-space. The state-space for the twisted CohFT is generated by $\{\phi_0,\dots,\phi_4\}$ with $\deg(\phi_i)=2i$. The state-space for the FJRW CohFT is 208-dimensional, generated by the elements $\{\phi_0,\dots,\phi_3,\gamma_1,\dots,\gamma_{204}\}$ where $\deg(\phi_i)=2i$ and $\deg(\gamma_i)=3$ (see, e.g. Chiodo--Ruan \cite{CR2}). Due to the existence of odd-degree classes, we must be careful about the pairing. We have
\[
\eta(\varphi^\alpha,\varphi_\alpha)=(-1)^{\deg(\varphi_\alpha)}\eta(\varphi_\alpha,\varphi^\alpha)=(-1)^{\deg(\varphi_\alpha)}.
\]
Thus,
\[
\int_{\M_{1,1}}\Omega_{1,1}^\star(\phi_0)\psi_1=\frac{1}{24}\omega_{1,1}^\star(\phi_0)=\frac{\chi^\star}{24}
\]
where
\[
\chi^\star=\begin{cases}
-200 & \star=\w\\
5 & \star=\lambda.
\end{cases}
\]

This concludes the proof of Lemma \ref{lem:onepoint}, and thus proves Theorem \ref{thm:mirrortheorem}.

\setcounter{section}{0}
\setcounter{equation}{0}
\renewcommand{\theequation}{\thesection.\arabic{equation}}
\setcounter{figure}{0}
\setcounter{table}{0}

\appendix
\section{}\label{appendix:genusoneformula}
In this appendix, we present a proof of Theorem \ref{Giv-DZ} using the Givental--Teleman reconstruction theorem. We begin by recalling the reconstruction theorem.

For the shifted CohFT $\Omega_{g,n}^\tau$, there is a corresponding \emph{topological field theory} $\omega_{g,n}^\tau$, which simply records the degree-zero part of the CohFT:
\[
\omega_{g,n}^\tau(\phi_{m_1-1}\cdots\phi_{m_n-1}):=\Omega_{g,n}^\tau(\phi_{m_1-1}\cdots\phi_{m_n-1})\cap H^0(\M_{g,n}).
\]
Roughly speaking, the reconstruction theorem states that $\Omega_{g,n}^\tau$ can be recovered from $\omega_{g,n}^\tau$ by a unique R-matrix action. More specifically, let $R(\mathbf u,z)$ be a matrix series as in Theorem \ref{thm:rmatrix}, and define
\[
T(\mathbf u,z):=z(1-R^{-1}(\mathbf u, z))\phi_0 = \cO(z^2)
\]
where
  \[
  R^{-1}(z)=1/R(z)=1-R_1z+\cO(z^2).
  \]
The T-matrix acts on the topological field theory $\omega_{g,n}^\tau$ to provide a new CohFT $T\omega_{g,n}^\tau$ defined by the following rule:
\[
T\omega^\tau_{g,n}(\phi_{m_1-1},\cdots,\phi_{m_n-1}):= \sum_{k\geq 0} \frac{1}{k!} (p_k)_* \omega^\tau_{g,n+k}\left(\phi_{m_1-1},\cdots,\phi_{m_n-1},T(\psi_{1}),\cdots , T(\psi_{k})\right),
\]
where $p_k:\M_{g,n+k}\rightarrow\M_{g,n}$ is the forgetful map and $\omega^\tau_{g,n+k}$ is linear with respect to the psi-classes.

The R-matrix, in turn,  acts on the CohFT $T\omega_{g,n}^\tau$ to provide a new CohFT $RT\Omega_{g,n}^\tau$ defined by
\begin{equation}\label{eq:A5}
RT\Omega_{g,n}(\phi_{m_1-1},\cdots,\phi_{m_n-1}) = \sum_{\Gamma \in G_{g,n}} \frac{1}{|\Aut (\Gamma)|} \Contr(\Gamma) ,
\end{equation}
where $G_{g,n}$ is the set of stable graphs of genus $g$ with $n$ legs, and for each $\Gamma\in G_{g,n}$, $\Contr(\Gamma)$ is given by the following construction:
\begin{itemize}
\item at each vertex $v$, we place $T\Omega_{g(v),n(v)}$;
  \item at each  leg $l$, we place $R^{-1}(\psi_{{l}})\phi_{m_l-1}$;
  \item at each edge $e=\{v_1,v_2\}$, we place $ V(\psi_{e_{v_1}},\psi_{e_{v_2}}) $
  where
\[
V(w,z) = \frac{\eta^{-1}-R^{-1}(w)\eta^{-1} R^{-1}(z)^t}{w+z} .
\]
\end{itemize}
The Givental--Teleman reconstruction theorem states the following.

\begin{theorem}[Teleman \cite{Teleman}]\label{thm:reconstruction}
There exists a unique R-matrix satisfying the conditions of Theorem \ref{thm:rmatrix} such that
\[
\Omega_{g,n}^\tau=RT\omega_{g,n}^\tau.
\]
\end{theorem}

As in the main body of the paper, we denote the unique R-matrix of Theorem \ref{thm:reconstruction} by $R^\lambda$, and we denote the corresponding T-matrix by $T^\lambda$. We now use Theorem \ref{thm:reconstruction} to prove the genus-one formula, which we restate as follows.

\begin{proposition}
We have
\[
\int_{\M_{1,1}}\partial_{u^\beta} \Omega^\tau_{1,0} = \, \frac{1}{2} (R_1^\lambda)_{\beta\beta}
   +\frac{1}{48} \sum_\alpha \partial_{u^\beta} \log \Delta_\alpha
\]
\end{proposition}

\begin{proof}
The integrand on the left-hand side of the proposition is equal to $\Omega_{1,1}^\tau(e_\beta)$, which, by Theorem \ref{thm:reconstruction}, can be computed as $R^\lambda T^\lambda\omega_{1,1}^\tau(e_\beta)$. We proceed by proving several lemmas, from which the proposition follows. For notational simplicity, we drop the $\lambda$ and $\tau$ from the superscripts.

\begin{lemma}
We have
\[
\int_{\M_{1,1}}\partial_{u^\beta} \Omega^\tau_{1,0}=\frac{1}{2}(R_1)_{\beta\beta}+\frac{1}{24}\left(\eta(R_1\phi_0,e^\beta)-\sum_\alpha\eta(R_1e_\beta,e^\alpha)\right).
\]
\end{lemma}

\begin{proof}[Proof of lemma]

To compute $\partial_{u^\beta} \Omega^\tau_{1,0}=\Omega_{1,1}^\tau(e_\beta)=R T\omega_{1,1}(e_\beta)$ by the reconstruction theorem, we sum over the stable graphs in $G_{1,1}$:
\[
\Gamma_1=\xy
(10,0); (20,0), **@{-};  (20.2,-0.2)*+{\bullet};
(21,-2)*+{\small{{}_{g=1}}};
\endxy
\quad \text{and} \quad
\Gamma_2=\xy
(10,0); (20,0), **@{-};  (20.2,-0.2)*+{\bullet};
(16,-2)*+{\small{{}_{g=0}}};
(24.6,0)*++++[o][F-]{}
\endxy
\]

The contribution of the first graph is
\begin{equation}\label{eq:graph1.1}
\Contr(\Gamma_1)=T\omega_{1,1}(R^{-1}(\psi_1)e_\beta)=T\omega_{1,1}(e_\beta)-T\omega_{1,1}(R_1e_\beta)\psi_1.
\end{equation}
We first compute the T-action on $\omega_{1,1}$:
\begin{align*}
T\omega_{1,1}(-)=\omega_{1,1}(-)+(p_1)_*\left(\omega_{1,2}(-,T_1)\psi_2^2\right)
\end{align*}
where $T_1=R_1\phi_0$. By the axioms of topological field theories, the values $\omega_{1,1}$ and $\omega_{1,2}$ can be computed by pairs-of-pants decompositions, and since $(p_1)_*\psi_2^2=\psi_1$, we obtain
\begin{align}\label{eq:graph1.2}
\nonumber T\omega_{1,1}(-)&=\sum_\alpha\left(\omega_{0,3}(e_\alpha,e^\alpha,-)+\omega_{0,4}(e_\alpha,e^\alpha,-,T_1)\psi_1 \right)\\
\nonumber&=\sum_\alpha\left(\langle\langle e_\alpha,e^\alpha,-\rangle\rangle_{0,3}+\sum_\gamma\langle\langle e_\alpha,e^\alpha,e_\gamma\rangle\rangle_{0,3}\langle\langle e^\gamma,-,T_1\rangle\rangle_{0,3}\psi_1 \right)\\
&=\sum_\alpha\left(\eta(e^\alpha,-)+\langle\langle e^\alpha,-,T_1\rangle\rangle_{0,3}\psi_1\right),
\end{align}
where we used the fact that $e_\alpha\bullet_\tau e^\alpha=e^\alpha$ in the last equality. Reinserting \eqref{eq:graph1.2} into \eqref{eq:graph1.1}, we obtain the following contribution from the first graph:
\begin{equation}\label{eq:graph1.3}
\Contr(\Gamma_1)=1+\eta(R_1\phi_0,e^\beta)\psi_1-\sum_\alpha\eta(R_1e_\beta,e^\alpha)\psi_1.
\end{equation}

The contribution of the second graph is
\begin{align*}
\Contr(\Gamma_2)&=T\omega_{0,3}(R^{-1}(\psi_1)e_\beta,V(\psi_2,\psi_3))[\mathrm{pt}]\\
&=\omega_{0,3}(e_\beta,V(\psi_2,\psi_3))[\mathrm{pt}].
\end{align*}
The constant term of $V(\psi_2,\psi_3)$ is the map $R_1\eta^{-1}:e_\alpha\rightarrow R_1e^\alpha$, which corresponds to the two-tensor $\sum_\alpha R_1e^\alpha\otimes e_\alpha$. Thus, we compute
\begin{align}\label{eq:graph2}
\nonumber\Contr(\Gamma_2)&=\sum_\alpha\omega_{0,3}(e_\beta,R_1e^\alpha,e_\alpha)[\mathrm{pt}]\\
\nonumber&=\eta(R_1e^\beta,e_\beta)[\mathrm{pt}]\\
&=(R_1)_{\beta\beta}.
\end{align}

Combining formulas \eqref{eq:graph1.3} and \eqref{eq:graph2}, taking into account the automorphism of $\Gamma_2$, and integrating, we have proved the lemma.
\end{proof}

\begin{lemma}
We have
\[
\sum_\alpha\Delta_\alpha\int_{\M_{0,4}}\Omega_{0,4}(e_\alpha,e_\alpha,e_\alpha,e_\beta)=\eta(R_1\phi_0,e^\beta)-\sum_\alpha\eta(R_1e_\beta,e^\alpha).
\]
\end{lemma}
\begin{proof}[Proof of lemma]
We can compute $\Omega_{0,4}(e_\alpha,e_\alpha,e_\alpha,e_\beta)= RT\omega_{0,4}(e_\alpha,e_\alpha,e_\alpha,e_\beta)$ as a sum over the stable graphs
\[
\Gamma_1=\xy
 (10,-0.1)*+{\bullet};
(15,5); (10,0), **@{-}; (15,-5); (10,0), **@{-};
(5,5); (10,0), **@{-}; (5,-5); (10,0), **@{-};
(15,0)*+{\small{{}_{g=0}}};
\endxy
\quad \text{and} \quad
\Gamma_2=\xy
(10,0); (20,0), **@{-}; (10,-0.1)*+{\bullet}; (20.2,-0.1)*+{\bullet};
(25,5); (20,0), **@{-}; (25,-5); (20,0), **@{-};
(5,5); (10,0), **@{-}; (5,-5); (10,0), **@{-};
(25,0)*+{\small{{}_{g=0}}};
(5,0)*+{\small{{}_{g=0}}};
\endxy\,,
\]
where there are three graphs of the second type, but we will soon see that they all give the same contribution.

The contribution of the first graph is
\begin{align*}
\Contr(\Gamma_1)&=T\omega_{0,4}(e_\alpha,e_\alpha,e_\alpha,e_\beta)\\
&\hspace{1cm}-3T\omega_{0,4}(R_1e_\alpha,e_\alpha,e_\alpha,e_\beta)\psi_1-T\omega_{0,4}(e_\alpha,e_\alpha,e_\alpha,R_1e_\beta)\psi_4\\
&=\omega_{0,4}(e_\alpha,e_\alpha,e_\alpha,e_\beta)+\omega_{0,5}(T_1,e_\alpha,e_\alpha,e_\alpha,e_\beta)\psi_1\\
&\hspace{1cm}-3\omega_{0,4}(R_1e_\alpha,e_\alpha,e_\alpha,e_\beta)\psi_1-\omega_{0,4}(e_\alpha,e_\alpha,e_\alpha,R_1e_\beta)\psi_4\\
&=\delta_{\alpha\beta}\Delta\alpha+\delta_{\alpha\beta}\eta(T_1,e_\alpha)\psi_1\\
&\hspace{1cm}-3\eta(R_1e_\alpha,e_\beta)\psi_1-\eta(e_\alpha,R_1e_\beta)\psi_4.
\end{align*}
Integrating, we obtain
\begin{equation}\label{eq:genuszerograph1}
\int\Contr(\Gamma_1)=\delta_{\alpha\beta}\eta(R_1\phi_0,e_\alpha)-3\eta(R_1e_\alpha,e_\beta)-\eta(e_\alpha,R_1e_\beta)
\end{equation}

The contribution from each of the second type of graph is given by
\begin{align*}
\Contr(\Gamma_2)&=\sum_\gamma\omega_{0,3}(e_\alpha,e_\alpha,e^\gamma)\omega_{0,3}(R_1e_\gamma,e_\alpha,e_\beta)[\mathrm{pt}]\\
&=\eta(e_\beta, R_1e^\alpha)[\mathrm{pt}].
\end{align*}
Integrating and summing the three graphs, we obtain
\begin{equation}\label{eq:genuszerograph2}
3\int\Contr(\Gamma_2)=3\eta(e_\beta, R_1e_\alpha).
\end{equation}
Combining Equations \eqref{eq:genuszerograph1} and \eqref{eq:genuszerograph2}, we obtain
\[
\int\Omega_{0,4}(e_\alpha,e_\alpha,e_\alpha,e_\beta)=\Delta_\alpha^{-1}\left(\delta_{\alpha\beta}\eta(R_1\phi_0,e^\alpha)-\eta(R_1e_\beta,e^\alpha)\right),
\]
and the lemma is proved.
\end{proof}

\begin{lemma}
We have
\[
\int_{\M_{0,4}}\Omega_{0,4}(e_\alpha,e_\alpha,e_\alpha,e_\beta)=-\frac{1}{2}\partial_{u_\beta}\Delta_\alpha^{-1}.
\]
\end{lemma}
\begin{proof}[Proof of lemma]
On the one hand, by definition of the pairing, we have
\[
\Delta_\alpha^{-1}=\eta(e_\alpha,e_\alpha)=\Omega_{0,3}(\phi_0,e_\alpha,e_\alpha)
\]
On the other hand, since the $e_\alpha$ are idempotent, we have
\[
\Delta_\alpha^{-1}=\eta(e_\alpha,e_\alpha)=\eta(e_\alpha,e_\alpha\bullet_\tau e_\alpha)=\Omega_{0,3}(e_\alpha,e_\alpha,e_\alpha).
\]
Differentiating each of these, we obtain
\[
\partial_{u_\beta}\Delta_\alpha^{-1}=2\Omega_{0,3}(\phi_0,e_\alpha,\partial_{u_\beta}e_\alpha)=3\Omega_{0,3}(e_\alpha,e_\alpha,\partial_{u_\beta}e_\alpha)+\Omega_{0,4}(e_\alpha,e_\alpha,e_\alpha,e_\beta)
\]
Using
\[
\Omega_{0,3}(\phi_0,e_\alpha,\partial_{u_\beta}e_\alpha)=\eta(e_\alpha,\partial_{u_\beta}e_\alpha)=\Omega_{0,3}(e_\alpha,e_\alpha,\partial_{u_\beta}e_\alpha),
\]
we finish the proof of the lemma.
\end{proof}

Finally, to complete the proof of the proposition, we apply the previous three lemmas sequentially:
\begin{align*}
\int_{\M_{1,1}}\partial_{u^\beta} \Omega^\tau_{1,0}&=\frac{1}{2}(R_1)_{\beta\beta}+\frac{1}{24}\left(\eta(R_1\phi_0,e^\beta)-\sum_\alpha\eta(R_1e_\beta,e^\alpha)\right)\\
&=\frac{1}{2}(R_1)_{\beta\beta}+\frac{1}{24}\sum_\alpha\Delta_\alpha\int_{\M_{0,4}}\Omega_{0,4}(e_\alpha,e_\alpha,e_\alpha,e_\beta)\\
&=\frac{1}{2}(R_1)_{\beta\beta}-\frac{1}{48}\sum_\alpha\Delta_\alpha\partial_{u_\beta}\Delta_\alpha^{-1}\\
&=\frac{1}{2}(R_1)_{\beta\beta}+\frac{1}{48}\sum_\alpha\partial_{u_\beta}\log\Delta_\alpha.
\end{align*}
\end{proof}

\bibliographystyle{alpha}
\bibliography{biblio}

\end{document}